\documentclass[twoside, 11pt]{article}

\usepackage{a4wide, amsmath, amsthm, amssymb, amscd, mathptmx}

\usepackage{amsfonts}

\usepackage{color}
\usepackage{pstricks}

\usepackage[all]{xy}\CompileMatrices\SelectTips{cm}{12}

\theoremstyle{plain}
\newtheorem{Thm}{\sc Theorem}[section]
\newtheorem{Theorem}[Thm]{\sc Theorem}
\newtheorem{Corollary}[Thm]{\sc Corollary}

\newtheorem*{Corollary*}{\sc Corollary}

\newtheorem{Proposition}[Thm]{\sc Proposition}
\newtheorem*{Proposition*}{\sc Proposition}
\newtheorem{Lemma}[Thm]{\sc Lemma}
\newtheorem{Conjecture}[Thm]{\sc Conjecture}

\theoremstyle{definition}
\newtheorem{Definition}[Thm]{Definition}

\theoremstyle{remark}
\newtheorem{Remark}[Thm]{Remark}

\newtheorem*{Example*}{Example}
\newtheorem*{Remark*}{Remark}

\newcommand{\EE}{{\mathbb E}}

\newcommand{\NN}{{\mathbb N}}
\newcommand{\ZZ}{{\mathbb Z}}

\newcommand{\PP}{{\mathbb P}}
\newcommand{\QQ}{{\mathbb Q}}
\newcommand{\RR}{{\mathbb R}}

\newcommand{\cA}{{\mathcal A}}

\newcommand{\cB}{{\mathcal B}}

\newcommand{\cS}{{\mathcal S}}
\newcommand{\cE}{{\mathcal E}}
\newcommand{\cF}{{\mathcal F}}
\newcommand{\cG}{{\mathcal G}}

\newcommand{\cK}{{\mathcal K}}
\newcommand{\cL}{{\mathcal L}}

\newcommand{\cO}{{\mathcal O}}
\newcommand{\cP}{{\mathcal P}}

\newcommand{\Pic}{{\mathop{\rm Pic \, }}}

\newcommand{\Coh}{{\mathop{\operatorname{Coh}\, }}}
\newcommand{\Refl}{{\mathop{\operatorname{Ref}\, }}}

\newcommand{\ch}{{\mathop{\rm ch \, }}}

\newcommand{\rk}{{\mathop{\rm rk \,}}}

\newcommand{\Sym}{{\mathop{{\rm Sym \, }}}}

\newcommand{\Spec}{{\mathop{{\rm Spec\, }}}}

\newcommand{\Vect}{{\mathop{{\rm Vect \,}}}}

\newcommand{\Sing}{{\mathop{{\rm Sing \,}}}}

\newcommand{\Div}{\mathop{\rm Div\, }}

\newcommand{\refl}{{\mathop{\operatorname{\rm ref} }}}

\begin{document}

\markboth {\rm }{}

\title{Intersection theory and Chern classes on normal varieties}
\author{Adrian Langer} \date{\today}

%%%
\maketitle
%%%

%%%%%%%%%%%%%%%%%%%%%%%%%%%%%%%%%%

{\noindent \sc Address:}\\
Institute of Mathematics, University of Warsaw,
ul.\ Banacha 2, 02-097 Warszawa, Poland\\
e-mail: {\tt alan@mimuw.edu.pl}

\medskip

\begin{abstract}
We study intersection theory and Chern classes of reflexive sheaves on normal varieties.
In particular, we define generalization of Mumford's intersection theory on normal surfaces to 
higher dimensions. We also define and study the second Chern class for reflexive sheaves on normal varieties. 
We use these results to prove some Bogomolov type inequalities on normal varieties in positive characteristic.
We also prove some new boundedness results  on normal varieties in positive characteristic.
\end{abstract}

\section*{Introduction}

Let $X$ be a normal projective variety of dimension $n$ defined over an algebraically closed field $k$.
It is well known how to intersect one Weil divisor $D$ with a collection of $(n-1)$ line bundles $L_1,...,L_{n-1}$.
The intersection number is simply the degree of the $0$-cycle $c_1(L_1)\cap ...\cap c_1(L_{n-1})\cap [D]$.
Mumford also  introduced a $\QQ$-valued intersection theory for two Weil divisors in case $X$ is a surface
(see \cite[Examples 7.1.16 and 8.3.11]{Fu}). To do so he passes to resolution of singularities $f: \tilde X\to X$, 
defines the pullback $f^*D$ of a Weil divisor $D$ as a $\QQ$-divisor and then intersects the pullbacks of Weil divisors on $\tilde X$. 

M. Enokizono in \cite[Appendix]{En} used a similar method for  higher-dimensional varieties, considering very special pullback, taking into account only exceptional divisors with codimension $2$ centers. His approach
uses the existence of resolution of singularities in the characteristic zero case and de Jong's alterations \cite{dJ} in positive characteristic.

We present a different approach to this result working directly on the singular variety. We make some computations in the Grothendieck group of $X$, similar to Kleiman's approach to intersection  numbers for Cartier divisors (see \cite[VI.2]{Ko}). Using some asymptotic Riemann--Roch formulas, we also reconstruct Mumford's  intersection theory on surfaces without passing to a resolution of singularities. The final outcome is the following result (see Section \ref{Section:intersection} for more precise results). 

\begin{Theorem}\label{main1}
For any Weil divisors $D_1$ and $D_2$ on $X$ there exists a unique $\ZZ$-multilinear symmetric form $\Pic X^{\times (n-2)}\to \QQ$, $(L_1,... ,L_{n-2})\to D_1.D_2.L_1...L_{n-2}$ such that:
\begin{enumerate}
\item If $D_1$ and $D_2$ are Cartier divisors then  
	$$ D_1.D_2.L_1...L_{n-2}= \int_X  c_1(\cO_X (D_1))\cap c_1(\cO_X(D_2))\cap c_1(L_1)\cap ...\cap c_1(L_{n-2})\cap [X].$$
\item If $D_2$ is a Cartier divisor then  
$$ D_1.D_2.L_1...L_{n-2}= \int_X c_1(\cO_X (D_2))\cap c_1(L_1)\cap ...\cap c_1(L_{n-2})\cap  [D_1]\in \ZZ.$$
\item If $L_1,...,L_{n-2}$ are very ample
then   for a general complete intersection surface  $S\in |L_1|\cap ...\cap |L_{n-2}|$ we have 
$$D_1.D_2.L_1...L_{n-2} =(D_1)_S.(D_2)_S.$$
\end{enumerate}
\end{Theorem}

The main aim of the paper is to determine whether one can similarly generalize the theory of Chern classes of vector bundles to reflexive sheaves on normal varieties. More precisely, if $X$ is smooth, then we have a well-defined Chern character from the Grothendieck group of vector bundles to the Chow ring. This is a homomorphism of rings that satisfies various remarkable properties, such as the Riemann--Roch theorem (see \cite[Chapter 15]{Fu}).  Using operational Chern classes, the Riemann--Roch theorem can also  be  proven for vector bundles on singular varieties (see \cite{BFM}) but there is no ring structure on the direct sum of the Chow groups.
However, if $X$ is a normal projective surface then one can use Mumford's intersection theory to define a rational Chow ring structure. In that case one can ask whether there exists a well-defined Chern character from the Grothendieck group of reflexive sheaves on $X$ to this rational (or real) Chow ring. 
We provide a definition of such a Chern character, which is conjecturally also a ring homomorphism (see Subsection \ref{K-theory-reflexive}).
To do so we revise the theory of Chern classes of reflexive sheaves on normal surfaces that was studied in \cite{La0} in the complex analytic setup. Here, we develop an algebraic approach that works in arbitrary characteristic.  In particular, we prove the following theorem (see Definition \ref{definition-c_2-surfaces} and Proposition \ref{global-basic-properties-c_2}).

\begin{Theorem}
For any normal proper algebraic surface $X$ defined over an algebraically closed field $k$ and  any 
coherent reflexive $\cO_X$-module $\cE$ on $X$ one can define its second Chern class $c_2 (\cE)\in A_0(X)\otimes \RR$
so that the following conditions are satisfied:
\begin{enumerate}
\item If $\cE$ has rank $1$ then  $c_2(\cE)=0$.
\item If $\cE$ is a vector bundle on $X$ then this class coincides with $c _2 (\cE)\cap [X].$
\item If  $ \pi: Y\to X$ is a finite morphism from a normal  surface $Y$ then
$$\int_Yc_2 ( \pi ^{[*]}\cE) = \deg \pi \cdot \int_Xc_2(\cE),$$
where $ \pi ^{[*]}\cE$ is the reflexive hull of $ \pi ^*\cE$.
\end{enumerate}
\end{Theorem}

In the above theorem $A_0(X)$ stands for the group of $0$-dimensional cycles modulo rational equivalence on $X$. Our approach to the above theorem is along similar lines as in \cite{La0} but we give simplified and more detailed versions of several proofs. We also obtain the following result about additional terms in the Riemann--Roch theorem (see Theorems \ref{bound-on-a(x,E)} and \ref{RR-theorem-easy}).

\begin{Theorem}\label{main2}
	Let $X$ be a normal proper algebraic surface defined over an algebraically closed field.
	Then there exists a constant $C$ depending only on $X$ such that for every rank $r$  coherent reflexive $\cO_X$-module $\cE$ on $X$ we have
	$$\left|\chi (X, \cE) -\left(\frac {1}{2}c_1 (\cE). (c_1(\cE) -K_X)- \int_Xc_2 (\cE) +r\chi (X, \cO_X)\right)\right|\le C r^2.$$
\end{Theorem}

We use this result to construct a good theory for the second Chern class (or character) for reflexive sheaves in higher dimensions. This works well in case of positive characteristic or for varieties with at most quotient singularities in codimension $\ge 2$ in characteristic $0$. In particular, our theory in the characteristic zero case generalizes the one developed in \cite[Section 3]{GKPT}. 

Here we state one of the main results in positive characteristic (see Theorem
\ref{properties-of-ch_2}). Let us mention that till now  the theory of Chern classes of reflexive sheaves on singular varieties that are well-behaved under coverings (e.g., Mumford's theory of Chern classes of $\QQ$-sheaves) was always restricted to the characteristic zero case.

\begin{Theorem}\label{main3}
Assume that $k$ has positive characteristic. For any normal projective variety $X/k$ and for any 
coherent reflexive $\cO_X$-module $\cE$ on $X$ there exists a $\ZZ$-multilinear symmetric 
form $\int _X c _2 (\cE): \Pic X ^{\times (n-2)}\to \RR$ such that:
\begin{enumerate}
\item If $\cE$ is a vector bundle on $X$ then $$\int _X c _2 (\cE)L_1...L_{n-2}=\int_X c _2 (\cE)\cap c_1(L_1)\cap ...\cap c_1(L_{n-2})\cap[X].$$
\item If $k\subset K$ is an algebraically closed field extension then 
$$\int _{X_K} c _2 (\cE_K)(L_1)_K...(L_{n-2})_K=\int _X c _2 (\cE)L_1...L_{n-2}.$$
\item If $n>2$ and  $ L_1$ is very ample then for a very general hypersurface $H\in |L_1|$ we have
$$\int _X c_2 (\cE)L_1...L_{n-2}=\int _{H}c_2 (\cE|_{H})L_2|_{H}...L_{n-2}|_{H}.$$
\end{enumerate}
\end{Theorem}

It is easy to see that this second Chern class is uniquely determined by equality
$$\int_Xc_2 ( F_X ^{[*]}\cE)L_1...L_{n-2} = p^2 \int_Xc_2(\cE)L_1...L_{n-2}$$
for the Frobenius morphism $F_X$,
 and  by the Riemann--Roch type inequality analogous to that from Theorem \ref{main2} (see Remark \ref{determinacy-c_2-higher-dim}). This suggests that one can define $\int _X \ch _2 (\cE)L_1...L_{n-2}$ as the limit
 $$\lim _{m\to \infty }\frac {  \chi (X, c_1(L_1)...c_1(L_{n-2}) \cdot {F_X^{[m]}\cE}  ) }{p^{2m}},$$
 where $c_1(L_1)...c_1(L_{n-2}) \cdot {F_X^{[m]}\cE} $ denotes the class in the Grothendieck group of coherent sheaves on $X$, obtained by intersecting the class of ${F_X^{[m]}\cE}$ with the product of first Chern classes of line bundles $L_1,..., L_{n-2}$ (see Subsection \ref{Grothendieck-group}).
The main problem in the proof is to show that such a limit exists.

\medskip

We give some applications of the obtained results. For example, we prove a general Bogomolov type inequality (see Theorem \ref{Bogomolov's-inequality}), Bogomolov's inequality for strongly semistable reflexive sheaves on normal varieties (see Corollary \ref{Bogomolov-strongly-ss}) and various boundedness results
(see, e.g.,  Corollary \ref{boundedness-4}). 
As a final application we show a quick proof of  certain boundedness theorem by Esnault and Srinivas on F-divided sheaves on normal varieties in positive characteristic (see \cite[Theorem 2.1]{ES}). 

Applications of the above theory to general restriction theorems and Bogomolov's inequality (also for Higgs sheaves) can be found in \cite{La4}. Other applications to non-abelian Hodge theory and Simpson's correspondence for singular varieties in positive characteristic are contained in \cite{La3}. After the first version of the paper was written, there appeared another application to  the abundance conjecture for threefolds in positive characteristic
(see \cite{Xu}).

\medskip

The structure of the paper is as follows. In the first section we gather some auxiliary results. In Section 2
we study intersection theory on normal varieties and in particular we prove Theorem \ref{main1}. In Section 3 we study relative Chern classes for resolutions of surface singularities. Then in Section 4 we use them to define Chern classes for reflexive sheaves on normal proper surfaces. In Section 5 we extend these results to higher dimensions proving Theorem \ref{main3}. Finally in Section 6 we give some applications of the obtained results.

\medskip

\subsection*{Notation}

Let $X$ be an integral normal  locally Noetherian scheme.
We say that an open subset $U\subset X$ is \emph{big} if every irreducible component of $X\backslash U$
has codimension $\ge 2$ in $X$.
A vector bundle on $X$ is a finite locally free $\cO_X$-module. The category of vector bundles on $X$ will be denoted by $\Vect (X)$.

If $X$ is an algebraic variety defined over some algebraically closed field $k$ then we write
$Z^1(X)$  for the group of Weil divisors on $X$. We also write $A_m(X)$ for the  group of $m$-dimensional cycles modulo rational equivalence on $X$ (see \cite[1.3]{Fu}) and $A^m(X)$ for $A_{
\dim X-m}(X)$.

\section{Preliminaries}

\subsection{Reflexive sheaves}\label{reflexive-sheaves}

In this subsection $X$ is an integral locally Noetherian scheme.
Let  $\Refl(\cO_X)$  be the category of coherent reflexive $\cO_X$-modules. It is a full subcategory of the category $\Coh(\cO_X)$ of coherent  $\cO_X$-modules. 
The inclusion functor $\Refl(\cO_X)\to \Coh (\cO_X)$ comes with a left adjoint $(\cdot )^{**}:\Coh(\cO_X)\to \Refl (\cO_X)$ given by the reflexive hull. The category $\Refl(\cO_X)$  is an additive category with kernels and cokernels and it comes with an associative and symmetric tensor product $\hat\otimes$ defined by
$$\cE\hat\otimes \cF:=(\cE\otimes _{\cO_X} \cF)^{**}.$$
The following well-known lemma can be found in \cite[Lemma 0EBJ]{SP}.

\begin{Lemma}\label{restriction-to-open} 
Let $j: U\hookrightarrow X $ be an open subscheme  with complement $Z$ such that the depth of $\cO_{X,z}$ is $\ge 2$ for all $z\in Z$. Then  $j_*$ and $j^*$ define adjoint equivalences of categories $\Refl (\cO_X)$ and $\Refl(\cO_U)$. In particular, the above assumptions are satisfied if $U$
is a big open subset of an integral locally Noetherian normal scheme $X$.  
\end{Lemma}

Let us also recall the following standard lemma (see, e.g., \cite[Tag 0EBF]{SP}).

\begin{Lemma}
Let $f: X\to Y$ be a flat morphism of integral locally Noetherian schemes. If $\cF$ is a coherent reflexive $\cO_Y$-module then $f^*\cF$ is reflexive on $X$.
\end{Lemma}

In general, pull back of a reflexive sheaf need not be reflexive.
If $f: X\to Y$ is a morphism between normal schemes and $\cE$
is a  coherent reflexive $\cO_Y$-module then we set
$$f^{[*]}\cE=(f^*\cE)^{**}.$$
If $\cF$ is a coherent reflexive $\cO_X$-module then we set
$$f_{[*]}\cF=(f_*\cE)^{**}.$$
Note that in general these operations define functors on reflexive sheaves that are not functorial with respect to morphisms.

\medskip
If $X$ is normal and $\cE\in \Refl (\cO_X)$ then we also use the notation
$$\Sym ^{[m]}\cE:= (\Sym ^m\cE)^{**}$$
for $\cE\in \Refl (\cO_X)$.

If $X$ is a normal scheme of finite type over an algebraically closed field $k$ of characteristic $p>0$ 
then $F_X:X\to X$ denotes the absolute Frobenius morphism.  Let us recall that by Kunz's theorem $F_X$ is flat only at regular points of $X$.
If $\cE\in \Refl (\cO_X)$
then for every $m $ we set $$F_X^{[m]} \cE:= (F_X^{m})^{[*]} \cE.$$

\medskip

The following well-known lemma is a corollary of Serre's criterion on normality.

\begin{Lemma}\label{reflexivization-iso-in-codim-1}
Let us assume that $X$ is normal. If $\cE$ is a coherent torsion free $\cO_X$-module then 
the canonical map $\cE\to \cE^{**}$ is an isomorphism in codimension $1$, i.e., it is injective and every irreducible component of the support of its cokernel has codimension $\ge 2$ in $X$.
\end{Lemma}

\medskip

 Let $X$ be an excellent integral normal scheme (this is the situation in which we will use the remark below). 
 Then the regular locus
 $U:=X_{\mathrm {reg}}$ is open  and the above lemma implies that $j^*: \Refl (\cO_X)\to \Refl(\cO_U)$ and 
 $j_*:  \Refl (\cO_U)\to  \Refl (\cO_X)$ are adjoint equivalences of categories, where $j: U\hookrightarrow X $ denotes the open embedding.
  
\subsection{Mumford's intersection theory on normal surfaces}\label{Mumford's-intersection}

Let $X$ be a normal proper surface defined over an algebraically closed field and let
 $f: \tilde X\to X$ be any resolution of singularities. For any Weil divisor $D$ on $X$ one can 
define its pull back $f^*D\in A^1(\tilde X)\otimes \QQ$ as  the only class of a $\QQ$-divisor 
for which $f_*f^*D=D$ and $\int_{\tilde X}[f^*D]\cap [E_i]=0$ for all exceptional curves $E_i$ (cf. \cite[Example 8.3.11]{Fu}).
For any two Weil divisors $D_1$ and $D_2$ we set 
$$[D_1]\cap [D_2]:= f_* ([f^*D_1]\cap [f^*D_2]).$$
This defines a $\ZZ$-bilinear symmetric form $A^1(X)\times A^1 (X)\to A_0(X)\otimes \QQ$, $(D_1,D_2)\to [D_1]\cap [D_2]$.
 In the following we write $D_1.D_2$ for the rational number $\int_X [D_1]\cap [D_2]$.
 
 \subsection{Bertini's theorem}
 
 Here we recall some Bertini type theorems. In particular, we have Bertini's theorem for smoothness, irreducibility and reducedness for unramified morphisms (see \cite[Theoreme 6.3 ]{Jo}) and Seidenberg's Bertini's theorem for normality for embeddings into the projective space (see \cite[Theorem 7']{Se}; see also \cite[Corollary 1.1.15]{HL} for a quick proof). However, we need these theorems in a slightly more general set-up and it is convenient to state them uniformly using \cite[Theorem 1 and Remark below Corollary 2]{CGM}. 
 
Let $\cP$ be one of the following properties of a locally noetherian scheme: being smooth, normal,  reduced or irreducible. Then we have the following result:

 \begin{Theorem}\label{Bertini}
 Let $X$ be a scheme of finite type over an algebraically closed field and let $\varphi: X\to \PP ^n_k$
be a morphism defined by a linear system $\Lambda$.
Let us assume that $\varphi$ has separably generated residue field extensions. 
If $X$ has property $\cP$ then there exists a nonempty Zariski open subset $U\subset \Lambda$ such that
every hypersurface $H\in \Lambda$ also has property $\cP$.
 \end{Theorem}
 
As usual we will quote this theorem by saying that a general $H\in \Lambda$ has property $\cP$.

 \subsection{Some results on Grothendieck's group of a normal variety}\label{Grothendieck-group}
 
Let $X$ be a normal projective variety of dimension $n$ defined over an algebraically closed field $k$.
Let $K(X)$ denote the Grothendieck's group of $X$. For any $0\le m\le n$ we denote by $K_m(X)$ the subgroup of $K(X)$ generated by classes of sheaves, whose support has dimension at most $m$. For a line bundle $L$ on $X$ we have an additive endomorphism
of $K(X)$ defined by
$$\cF\to c_1(L)\cdot \cF:= \cF- L^{-1}\otimes \cF$$
(see \cite[Chapter VI, Definition 2.4]{Ko}). 

 \medskip
 
Let us recall the following result (see  \cite[Chapter VI, Proposition 2.5]{Ko} and its proof):

\begin{Lemma}\label{Kollar}
\begin{enumerate}
\item $c_1(L)\cdot K_m(X)\subset K_{m-1}(X)$ for every $m$.
\item $c_1(L_1)$ and $c_1 (L_2)$ commute for any two line bundles $L_1$ and $L_2$.
\item  For any two line bundles $L_1$ and $L_2$ we have equality of endomorphisms
$$c_1(L_1\otimes L_2)= c_1(L_1)+c_1(L_2)-c_1(L_1)\cdot c_1(L_2).$$
\item If $Y\subset X$ is integral  and $L|_Y\simeq \cO_Y (D)$ for some effective 
Cartier divisor $D$ on $Y$ then $c_1(L)\cdot \cO_Y=\cO_D$.
\end{enumerate}
\end{Lemma}

If $D$ is  a Weil divisor on $X$ then we write $[\cO_D]$ for the class of $\cO_X-\cO _X(-D)$ in $K(X)$. 
We will need the following lemmas:

\begin{Lemma}\label{difference}
If $D$ is a Weil divisor on $X$ then $[\cO_D]\in K_{n-1} (X)$.
Moreover, if  we write $D=\sum a_i D_i$ for some   Weil divisors $D_i$  and some integers $a_i$ then the class
$[\cO_D]- \sum  a_i [\cO_{D_i}]$ lies in $K_{n-2}(X)$.
\end{Lemma}

\begin{proof}
Since $X$ is normal,  Weil divisors are Cartier outside of a closed subset of codimension $\ge 2$. 
But for any closed subscheme $Y\subset X$ we have a short exact sequence
$$K(Y)\to K(X)\to K(X\backslash Y)\to 0.$$
This together with Lemma \ref{Kollar}, (3) shows that $[\cO_D]- \sum a_i [\cO_{D_i}]$ lies in $K_{n-2}(X)$.

To prove the first part of the lemma let us write $D$ as $D=D_1-D_2$ for some  effective Weil divisors $D_1$  
and $D_2$. Since $D_i$, $i=1,2$ are effective, we have short exact sequences
$$0\to \cO _X(-D_i) \to \cO _X \to \cO _{D_i}\to 0$$
showing that $[\cO_{D_i}]\in K_{n-1} (X)$. Now the first part of the proof implies that $[\cO_D]- [\cO_{D_1}]+[\cO_{D_2}]\in K_{n-2} (X)$, so  $[\cO_D]\in K_{n-1} (X)$.
\end{proof}

\medskip

\begin{Lemma}\label{reflexive-in-K(X)}
Let $\cF$ be a coherent  torsion free $\cO_X$-module of rank $r$. Then the class $\alpha (\cF)$ of  
$\cF- \cO_X^{\oplus r}+[\cO_{-c_1 (\cF)}]$ lies in $K_{n-2}(X)$.
\end{Lemma}

\begin{proof}
If  $r=1$ and $\cF$ is reflexive then by definition of $c_1 (\cF)$ we have
$\alpha (\cF)=0$ in $K(X)$.  In general, if $r=1$ then
$$\alpha (\cF)=\alpha (\cF)-\alpha (\cF ^{**})= \cF- \cF ^{**}, 
$$
$$\cF- \cO_X^{\oplus r}+[\cO_{-c_1 (\cF)}]= \cF- \cO_X^{\oplus r}+[\cO_{-c_1 (\cF)}] -(\cF ^{**}- \cO_X^{\oplus r}+[\cO_{-c_1 (\cF)}])= \cF- \cF ^{**}, 
$$
so the assertion follows from the fact that the cokernel of
$\cF\to \cF^{**}$ is supported in codimension $\ge 2$.

If $r>1$ then $\cF$ has a filtration $\cF=\cF^0\supset \cF^1\supset ...\supset \cF ^r=0$, whose quotients $\cL_i=\cF^i/\cF^{i+1}$ are rank $1$ torsion free sheaves.
  Then we have
  $$\alpha (\cF)= [\cO _{-c_1(\cF)} ]- \sum  [\cO _{-c_1(\cL _i)} ]+  
  \sum \alpha (\cL _i).$$
  Since $c_1(\cF) =\sum c_1(\cL_i)$ Lemma \ref{difference} shows that $ [\cO _{-c_1(\cF)} ]- \sum  [\cO _{-c_1(\cL _i)} ]$
  lies in $K_{n-2}(X)$. Therefore the assertion follows from the rank $1$ case.
\end{proof}

\section{Intersection theory on normal varieties}\label{Section:intersection}

Let $X$ be a normal projective variety of dimension $n$ defined over an algebraically closed field $k$.
We  write $N^1(X)$ for the group of line bundles on $X$ modulo numerical equivalence. By the N\'eron--Severi theorem of the base, $N^1(X)$ is a finitely generated free $\ZZ$-module.

\subsection{Intersection of a Weil divisor with Cartier divisors}

In this subsection we define a K-theoretic intersection of a Weil divisor with Cartier divisors
and compare it with standard definition using intersections with Chern classes of line bundles.

\begin{Lemma}\label{intersection-with-Weil-div}
 The image of the map $Z^1(X)\times \Pic X ^{\times (n-1)}\to K(X)$ defined by 
$$(D, L_1,..., L_{n-1})\to  
c_1(L_1)\cdot ...\cdot c_1(L_{n-1}) \cdot [\cO_D]$$
is contained in $K_0(X)$. Moreover, this map is $\ZZ$-linear with respect to each variable and
symmetric with respect to $(L_1,..., L_{n-1})$  for fixed Weil divisor $D$.
\end{Lemma}

\begin{proof}
Let us  write  $D=\sum a_iD_i$ for some  prime Weil divisors $D_i$ and some integers $a_i$. 
 By Lemma \ref{Kollar}, (1) we have 
$c_1(L_1)...c_1(L_{n-1}) K_{n-1}(X)\subset K_0(X)$ and $c_1(L_1)...c_1(L_{n-1}) K_{n-2}(X)=0$. Therefore Lemma 
\ref{difference} implies the first assertion. It also shows that
$$c_1(L_1)...c_1(L_{n-1})[\cO_D]= \sum a_i 
c_1(L_1)...c_1(L_{n-1})[\cO_{D_i}].
$$
This proves that the map is $\ZZ$-linear with respect to $D$.
By Lemma \ref{Kollar}, (2) the map is symmetric for fixed $D$. Moreover,  we have
\begin{align*}
&c_1(L_1\otimes M_1)c_1(L_2)...c_1(L_{n-1}) \cdot [\cO_D]-c_1(L_1)...c_1(L_{n-1}) \cdot [\cO_D]-c_1(M_1)c_1(L_2)...c_1(L_{n-1}) \cdot [\cO_D]\\
&=-c_1(L_1)c_1(M_1)c_1(L_2)...c_1(L_{n-1}) \cdot [\cO_D]=0,
\end{align*} 
which finishes the proof that the map is $\ZZ$-linear with respect to all variables.
\end{proof}

\medskip

We  also have  a map $\Pic X ^{\times (n-1)}\to A_0(X)$ defined by 
$$(L_1,..., L_{n-1})\to  
c_1(L_1)\cap ...\cap c_1(L_{n-1}) \cap [D]$$
(see \cite[Section 2.5]{Fu}). This map is also symmetric and multilinear (see \cite[Proposition 2.5]{Fu}).

To compare the above maps we can use the map $\psi: K_0(X)\to A_0(X)$ given by sending $\cF$ to $\sum _{x\in X(k)} l_x(\cF)\, [x]$, where $l_x (\cF)$
is the length of $\cF_x$ as an $\cO_{X.x}$-module (cf. \cite[Example 18.3.11]{Fu}, where an analogous map is defined also for cycles of higher dimension but it goes into the Chow group with rational coefficients). This map is an isomorphism with the inverse
$\varphi: A_0(X)\to K_0(X)$ given by sending $[x]$ to $\cO_ x$.

\medskip

The following lemma, generalizing the Riemann--Roch theorem on curves, compares the above maps.

\begin{Lemma}
For any Weil divisor $D$ and any line bundles $L_1$, ..., $L_{n-1}$ we have
$$c_1(L_1)\cap ...\cap c_1(L_{n-1}) \cap [D] = \psi( c_1(L_1)...c_1(L_{n-1}) \cdot [\cO_D]).$$
\end{Lemma}

\begin{proof} By the above we know that both sides of our equality are $\ZZ$-linear in $L_i$.
But any line bundle $L$ can be written as $A\otimes B^{-1}$ for some very ample line bundles $A$ and $B$. So it is sufficient to prove the above equality assuming that all $L_i$ are very ample.

For general divisors $H_i\in |L_i|$ the complete intersection $C=H_1\cap ...\cap H_{n-1}$ is a smooth curve
and $D\cap C$ is a $0$-cycle representing $c_1(L_1)\cap ...\cap c_1(L_{n-1}) \cap [D]$.
But by  Lemma \ref{Kollar}, (4) we have equality  $c_1(L_1)...c_1(L_{n-1})[\cO_D]=\cO _{D\cap C}$ in $K(X)$, so the required assertion is clear.
\end{proof}

Note that the above defined maps descend to the intersection product $A^1(X)\times (\Pic X )^{\times (n-1)}\to A_0(X)$ given by
$(D, L_1,..., L_{n-1})\to  c_1(L_1)\cap ...\cap c_1(L_{n-1}) \cap [D]$. 
Since $\chi (X, \cdot)= \int _X\circ \psi$,  the above lemma  shows that this descends to an intersection product $A^1(X)\times N^1(X)^{\times (n-1)}\to \ZZ$ given by
$$D.L_1...L_{n-1}:= \chi (X, c_1(L_1)...c_1(L_{n-1}) \cdot [\cO_D])= \int_X c_1(L_1)\cap ...\cap c_1(L_{n-1}) \cap [D].$$

\subsection{Intersection of two Weil divisors with Cartier divisors}\label{intersecting-2-Weil}

In this subsection we define intersection number for two Weil divisors and a collection of Cartier divisors.
We assume that $n\ge 2$.

\begin{Proposition}\label{limit-square-of-divisor}
Let $D$ be a Weil divisor and let $L_1,..., L_{n-2}$ be line bundles on $X$.
Then 
the sequence $$\left(\frac {  \chi (X, c_1(L_1)...c_1(L_{n-2}) \cdot [\cO_{mD} ]) }{m^2} \right)_{m\in \NN}$$
 is convergent and its limit is a rational number. 
 \end{Proposition}

\begin{proof}
 By Lemma \ref{Kollar} (3) we have
\begin{align*}
&{  \chi (X, c_1(L_1\otimes M_1)c_1(L_2)...c_1(L_{n-2}) \cdot [\cO_{mD}] ) }
={  \chi (X, c_1(L_1)...c_1(L_{n-2}) \cdot [\cO_{mD}] ) }\\
&+ {  \chi (X, c_1(M_1)c_1(L_2)...c_1(L_{n-2}) \cdot [\cO_{mD}] ) }- 
 {  \chi (X, c_1(M_1)c_1(L_1)...c_1(L_{n-2}) \cdot [\cO_{mD}] ) }.
\end{align*}
Since
\begin{align*}
\lim _{m\to \infty}\frac{\chi (X,  c_1(M_1)c_1(L_1)...c_1(L_{n-2}) \cdot [\cO_{mD}])}{m^2}
=\lim _{m\to \infty}\frac{m D.  c_1(M_1)c_1(L_1)...c_1(L_{n-2} )}{m^2}=0,
\end{align*}
we see that if  $\lim _{m\to \infty }\frac {  \chi (X, c_1(L_1)...c_1(L_{n-2}) \cdot [\cO_{mD}] ) }{m^2} $ and 
$\lim _{m\to \infty }\frac {  \chi (X, c_1(M_1)c_1(L_2)...c_1(L_{n-2}) \cdot [\cO_{mD}] ) }{m^2}$
exist then $$\lim _{m\to \infty }\frac {  \chi (X, c_1(L_1\otimes M_1)c_1(L_2)...c_1(L_{n-2}) \cdot [\cO_{mD}] ) }{m^2}$$
exists and it is equal to the sum of these limits.

Since any line bundle $L$ can be written as $A\otimes B^{-1}$ for some very ample line bundles $A$ and $B$ and the formula in the sequence is symmetric in $(L_1, ..., L_{n-2})$, it suffices to prove the existence of the (rational) limit assuming that all line bundles $L_i$ are very ample. Moreover, since 
$$\lim _{m\to \infty } \frac {  \chi (X, c_1(L_1)...c_1(L_{n-2}) \cdot [\cO_{mD} ]) }{m^2}=
-\lim _{m\to \infty }\frac {  \chi (X, c_1(L_1)...c_1(L_{n-2}) \cdot \cO_X (-mD)) }{m^2}
$$
it is enough to show the existence of the latter limit.

Let $K$ be any uncountable algebraically closed field containing $k$ and let $X_K\to X$ be the base change. Since the
Euler characteristic does not change under field extesion, it is sufficient to prove convergence of the sequence in question after base change to $K$. Therefore, in the following, we may assume that the base field $k$ is uncountable. 

Let us write $D= \sum d_i D_i$ as an integral combination of prime Weil divisors $D_i$. By Theorem \ref{Bertini} for a general sequence $H_i\in  |L_i| $, $i=1,...,n-2,$  each intersection  $X_j:=\bigcap _{i\le j} H_i $ is irreducible and normal,  and all $D_i\cap X_j$ have codimension $1$ in $X_j$. So $D_{X_j}=\sum d_i (D_i\cap X_j)$ is a well defined Weil divisor on $X_j$. Let us fix $m\in \ZZ$. Then for a general sequence as above and any $1\le j\le n-2$,
the restriction of  $\cO_{X_{j-1}}(mD_{X_{j-1}})$ to $X_j$ is reflexive (see  \cite[Corollary 1.1.14]{HL}) and hence isomorphic to $\cO_{X_{j}}(mD_{X_{j}})$. 
In that case we have short exact sequences
$$0\to  \cO_{X_{j-1}}(mD) \otimes L_{j}^{-1}\to \cO_{X_{j-1}}(mD_{X_{j-1}}) \to \cO_{X_{j}}(mD_{X_{j}})\to 0.$$ 
These sequences show that $\cO_{X_{j}}(mD_{X_{j}}) =c_1(L_j)\cdot \cO_{X_{j-1}}(mD_{X_{j-1}})$
in $K(X)$. So if we set $S:=X_{n-2}$  then  $\cO_{S}(mD_{S}) =c_1(L_1)...c_1(L_{n-2}) \cdot \cO_X({mD}) $ in $K(X)$.
This equality holds for all $m\in \ZZ$ if the sequence $H_1,...,H_{n-2}$ is very general. 
Now by the Riemann--Roch theorem  on $S$ (see Theorems \ref{RR-theorem-easy} and \ref{bound-on-a(x,E)}) we get
$$
\lim _{m\to \infty }\frac {  \chi (X, c_1(L_1)...c_1(L_{n-2}) \cdot \cO_X (-mD)) }{m^2}
=\lim _{m\to \infty}\frac{\chi ({S},  \cO _{S} (-mD_S))}{m^2}=  \frac{1}{2}D_{S}^2,
$$
where $D_{S}^2$ is the self intersection of $D_S$ in the sense of Mumford's intersection pairing on $S$.
\end{proof}

\medskip

For any Weil divisor $D$ and any line bundles $L_1,...,L_{n-2}$ we set
$$D^2.L_1...L_{n-2}:= 2 \lim _{m\to \infty} \frac{  \chi (X, c_1(L_1)...c_1(L_{n-2}) \cdot [\cO_{mD}]) }{m^2}. $$

\begin{Theorem}\label{properties-of-intersection-product}
Let us consider the map $Z^1(X)\times Z^1(X)\times \Pic X^{\times (n-2)}\to \QQ$ defined by sending
$(D_1,D_2,L_1,... ,L_{n-2})$ to
$$ D_1.D_2.L_1...L_{n-2}:= \frac{1}{2}\left((D_1+D_2)^2.L_1...L_{n-2}   -D_1^2.L_1...L_{n-2}-D_2^2.L_1...L_{n-2}\right).$$
This map satisfies the following conditions:
\begin{enumerate}
\item It is $\ZZ$-linear in each variable. 
\item It is symmetric in $D_1$ and $D_2$.
\item It is symmetric in $L_1, ... ,L_{n-2}$.
\item If $D_2$ is a Cartier divisor then  $ D_1.D_2.L_1...L_{n-2}= D_1.\cO_X(D_2) L_1...L_{n-2}$.
\item If we fix $D_1, D_2\in  Z^1 (X)$ and assume that  $n>2$ and $L_1$ is very ample
then  for a very general hypersurface  $H_1\in |L_1|$ we have 
$$D_1.D_2.L_1...L_{n-2} =(D_1)_{H_1}.(D_2)_{H_1}.L_2|_{H_1}...L_{n-2}|_{H_1}.$$
\end{enumerate}
\end{Theorem}

\begin{proof}
The fact that the map is $\ZZ$-linear in $L_i$ follows from the first part of the proof of Proposition
\ref{limit-square-of-divisor}. 
If we fix $D_1, D_2\in  Z^1 (X)$ and assume that  all $L_1,...,L_{n-2}$ are very ample
then  for a very general complete intersection surface  $S\in |L_1|\cap ...\cap |L_{n-2}|$ we have 
$$D_1.D_2.L_1...L_{n-2} =(D_1)_S.(D_2)_S.$$
If $D_1=D_2$ this follows from the proof of Proposition \ref{limit-square-of-divisor}
and the the general case can be reduced to this using the equality
$$ (D_1)_S. (D_2)_S= \frac{1}{2}\left((D_1+D_2)_S^2   -(D_1)_S^2-(D_2^2)_S\right).$$
Now (5) follows from the above mentioned linearity, which allows us to assume  that also $L_2,...,L_{n-2}$ are very ample.
(2) is clear and (3) follows from Lemma \ref{Kollar}, (2). To prove that the map is $\ZZ$-linear in $D_1$ and $D_2$ we can assume that $L_1,...,L_{n-2}$ are very ample. In this case the assertion follows from (5). In the same way we can reduce
(4) to the normal surface case, where the assertion is clear.
\end{proof}

The above theorem implies that we have a well defined induced intersection form
$$A^1(X)\times A^1(X)\times N^1(X)^{\times (n-2)}\to \QQ ,$$
generalizing Mumford's intersection pairing on surfaces. Note that our approach reconstructs Mumford's intersection pairing on normal surfaces without using any resolution of singularities.

\medskip

\medskip

Now let fix very ample line bundles $L_1,...,L_{n-2}$ and consider 
 a $\QQ$-valued intersection pairing $\langle \cdot, \cdot \rangle : A^1(X)\times A^1(X)\to \QQ$ by setting
$$\langle D_1, D_2\rangle := D_1.D_2.L_1...L_{n-2} .$$
Let us write $N_L(X)$ for the quotient of $A^1 (X)$ modulo the radical of this intersection pairing. 
Then we have an induced non-degenerate intersection pairing
$$\langle \cdot, \cdot \rangle : N_L(X)\times N_L(X)\to \QQ.$$
If $n=2$ then we write $N(X)$ instead of $N_L(X)$.

\begin{Lemma}\label{finite-generation}
$N_L(X)$ is a free $\ZZ$-module of finite rank. In particular, there exists a positive integer $N$ such that the intersection pairing $\langle \cdot , \cdot \rangle$ takes values in $\frac{1}{N}\ZZ\subset \QQ$.  If  $\rk N_L (X)=s$ then the intersection pairing $\langle \cdot , \cdot \rangle$ has signature $(1, s-1)$.
\end{Lemma} 

\begin{proof}
By \cite[Lemma 1.18]{La4} the intersection pairing induces a $\QQ$-valued intersection pairing $\langle \cdot, \cdot \rangle : B^1(X)\times B^1(X)\to \QQ$, where $B^1(X)$ is the group of algebraic equivalence classes of Weil divisors on $X$. Since $N_L(X)$  is the quotient of $B^1 (X)$ modulo the radical of this intersection pairing, $N_L(X)$ is a free $\ZZ$-module of finite rank by the theorem of the base. Since the product is $\ZZ$-linear in both variables, its image in $\QQ$ is a $\ZZ$-submodule. This implies  existence of  $N$  in the lemma’s statement. 
Let $D$ be a Weil divisor on $X$ and $H$ an ample line bundle on $X$. By Theorem \ref{properties-of-intersection-product}, (5) and the Hodge index theorem on normal surfaces we have 
\[D^2.L_1...L_{n-2} \cdot H^2.L_1...L_{n-2}  \le (D.H.L_1...L_{n-2})^2.
\]
This shows that the proof of \cite[Lemma 1.19]{La4} works. This lemma implies the last assertion.  
\end{proof}

\subsection{Comparison with the de Fernex--Hacon pullback and Enokizono's intersection numbers}\label{deFernex-Hacon}

In this subsection we compare the obtained intersection theory with the one obtained by intersecting de Fernex--Hacon pullbacks of Weil divisors. This subsection is not needed in the following but we will need to use some results and notation introduced in further part of the paper.

Let us recall that if $X$ is a normal variety defined over some algebraically closed field $k$ and $f: Y\to X$ is a birational morphism from a normal variety then for any Weil divisor $D$ on $X$ in \cite[Definition 2.5]{dFH} the authors define $f^{\sharp}D$ so that $\cO_X (-f^{\sharp} D)=f^{[*]}\cO_X(-D)$.
If $f: \tilde X\to X$ is a resolution of singularities and $X$ is a proper normal surface then in
the notation of Subsection \ref{pullback} we have
$$f^*(-mD)=- f^{\sharp} (mD)-\sum _{x\in f(E)} c_1 (f_x, f^{[*]} \cO_X (-mD)), $$ 
where on the left hand side we have Mumford's pullback.
But by Theorem \ref{c_1-full-sheaves} there are only finitely many possibilities for $c_1 (f_x, f^{[*]} \cO_X (-mD))$,
so by $\ZZ$-linearity of Mumford's pullback we have
$$f^*(D)=\lim _{m\to \infty }\frac{f^{\sharp} (mD)}{m}.$$
This shows that, in the surface case, the de Fernex--Hacon pullback of Weil divisors defined in \cite[Definition 2.9]{dFH}
coincides with Mumford's pullback. In the characteristic zero case this fact was known (see \cite[Section 2]{BdFF}).

In higher dimensions (even in characteristic $0$)  the de Fernex--Hacon pullback  satisfies $f^*(-D)\ne -f^*D$, so it is  not useful for defining intersection form on $X$. If one wants to get pullback of Weil divisors similar to Mumford's rational pullback, one needs to consider only very special morphisms (see \cite{Sc}). 
One can also consider a partial (Mumford's) pullback 
$$f^*_MD:=\sum _{{\rm codim}\, f(E)\le 2}a(E)\cdot E,$$
where $a(E)$ is the coefficient of $E$ in the de Fernex--Hacon pullback   $f^*D$, and the sum is taken over all
prime divisors $E$ on $\tilde X$ with centers of codimension $1$ or $2$.  Note that even if $X$ is smooth, the above pullback does not coincide with the usual pullback of Cartier divisors. However, this partial pullback can be used to define intersection of two Weil divisors with Cartier divisors.

In  \cite[Appendix]{En}  Enokizono used this pullback to define intersection numbers as follows.
Using de Jong's results \cite{dJ} one can find an alteration $f: \tilde Y\to X$ from a smooth projective variety $\tilde Y$. Let $\tilde Y\stackrel{g}{\longrightarrow}Y\stackrel{\pi}{\longrightarrow}X$ denote the Stein factorization of $f$, where $Y$ is normal, $g$ is birational and $\pi$ is finite. Assume that the dimension $n$ of $X$ is at least $2$.
Then for any Weil divisors $D_1$, $D_2$ and line bundles $L_1,..., L_{n-2}$
one can define the intersection number
$$(D_1.D_2.L_1...L_{n-2})_E:= \frac{1}{\deg \pi} \int_{\tilde Y}g_M^*(\pi^*D_1)g_M^*(\pi^*D_2)f^*L_1...f^*L_{n-2}.$$ 
This number is independent of the choice of alterations (see \cite[Lemma A.5]{En}) so in particular it satisfies the following lemma.

\begin{Lemma}\label{c_2-direct-sum-any-dim-Enokizono} 
	Let $D_1, D_2$ be Weil divisors and $L_1,..., L_{n-2}$ line bundles on $X$. 
	Let  $ \pi: Y\to X$ be a finite morphism from a normal  projective variety $Y$. Then we have
	$$(\pi^*D_1.\pi^*D_2. \pi^*L_1...\pi^*L_{n-2})_E= \deg \pi \cdot (D_1.D_2.L_1...L_{n-2})_E.$$
\end{Lemma}

The following lemma follows easily from the numerical characterization of  Mumford's pullback on normal surfaces (the same as for resolution of singularities in Subsection \ref{Mumford's-intersection}).

\begin{Lemma}
Let  $f: \tilde Y\to X$ be an alteration between normal projective varieties with Stein factorization 
\[
\tilde{Y} \xrightarrow{g} Y \xrightarrow{\pi} X.
\]
Let $L$ be a very ample line bundle on $X$ and let $D$ be a Weil divisor on $X$. For a general hyperplane $H\in |L|$ let 
\[
\nu: B \to \pi^{-1}(H) \quad \text{and} \quad \tilde{\nu}: \tilde{B} \to f^{-1}(B)
\]
be the normalizations. Let
\[
\tilde{g}: \tilde{B} \to B \quad \text{and} \quad \tilde{\pi}: B \to H
\]
denote the maps induced by \( g \) and \( \pi \), respectively. Then  we have 
\[\tilde \nu^* (g_M^*(\pi^*D))= \tilde g_M^*(\pi^*(D|_H)).
\]
\end{Lemma}

Note that we need to take the normalizations of both $f^{-1}(B)$ and $\pi^{-1}(H)$ as
one can construct examples of finite morphisms $\pi: Y\to X$ for which the preimage $\pi^{-1}(H)$ is non-normal for every member $H$ of a very ample linear system on $X$. 
The above lemma immediately implies the following version of \cite[Theorem A.1, (iv)]{En} (whose proof was skipped by the author).

\begin{Corollary}
Let $D_1, D_2$ be Weil divisors and $L_1,..., L_{n-2}$ line bundles on $X$. Assume that $n>2$ and $L_1$ is very ample. Then for a general hypersurface  $H_1\in |L_1|$ we have 
$$(D_1.D_2.L_1...L_{n-2})_E =((D_1)_{H_1}.(D_2)_{H_1}.L_2|_{H_1}...L_{n-2}|_{H_1})_E.$$
\end{Corollary}

Since both Enokizono's and our intersection numbers agree on normal surface and they do not change when taking a base change to larger algebraically closed field,
Theorem \ref{properties-of-intersection-product}, (5) and the above corollary imply the following.

\begin{Corollary}
Let $D_1, D_2$ be Weil divisors and $L_1,..., L_{n-2}$ line bundles on $X$. Then 
$$(D_1.D_2.L_1...L_{n-2})_E=D_1.D_2.L_1...L_{n-2}.$$
\end{Corollary}

\section{Local relative Chern classes for resolutions of normal surfaces}

In this section we revise the theory of local relative Chern classes  for resolutions of normal surfaces.
This theory was studied in \cite{Wa} in the rank $2$ case and in  \cite{La0} in general, but only in the complex analytic setup. Here we develop an algebraic approach in an arbitrary characteristic.  Our approach is along similar lines 
as in \cite{La0} but we give simplified and more detailed versions of several proofs. 

\medskip

Let $k$ be an algebraically closed field and let $A$ be an excellent normal $2$-dimensional Henselian local $k$-algebra. The developed theory seems to work in more general situations as in \cite{Li} but we will need it only for henselizations of local rings of algebraic surfaces defined over an algebraically closed field.

Let $X=\Spec A$ and let $x\in X$ be the closed point of $X$. 
Let  $f:\tilde X\to X$ be a desingularization of $X$, i.e., a proper birational morphism from a regular surface $\tilde X$. Here a surface is a reduced noetherian separated $k$-scheme of dimension $2$ but we do not assume that it is of finite type over $k$. Let $E$ be the exceptional locus of $f$ considered with the reduced scheme structure. 
One can also assume that $f$ is \emph{good}, i.e., $E$ is a simple normal crossing divisor, but this will not be used in the following. 

\subsection{Local relative Chern classes of vector bundles for resolutions of surfaces}\label{local-relative-RR} 

Let $\cF$ be a vector bundle on $\tilde X$. If $\{E_i\}$ denote the irreducible components of $E$ then 
 the intersection matrix $[E_i.E_j]$ of the exceptional divisor is negative definite (see, e.g., \cite[Lemma 14.1]{Li}).
So  there exists a unique $\QQ$-divisor $c_1(f, \cF )$ supported on $E$ such that for every irreducible component $E_{i}$ of $E$ we have 
$$c_1(f ,\cF ) . E_{i} =\deg \cF |_{E_{i}}.$$
We call $c_1(f, \cF )$ the \emph{first relative Chern class} of $\cF$ with respect to $f$.

Let $\Div _E(\tilde X)$ denote the group of divisors on $\tilde X$ that are supported on $E$.
Then $c_1(f, \cF )$ is an element of $\Div _E(\tilde X)\otimes \QQ$. Since the canonical 
map $\Div _E(\tilde X)\to \Pic \tilde X$ is injective, we can, without any loss of information, consider  $c_1(f, \cF )$ as an element of  $ \Pic \tilde X\otimes \QQ$.

\medskip

Let $\tilde \pi: \tilde Y\to \tilde X$ be a generically finite proper morphism from a regular surface $\tilde Y$.
Let us consider the Stein factorization of $f\circ \tilde \pi$  into a proper birational morphism $g:\tilde Y\to Y$ and a finite morphism $\pi:Y\to X$. Clearly, $Y\to X$ corresponds to a finite extension $A\to B$ with $B$ a normal domain. Since $A$ is Henselian, $B$ is local and also Henselian. Since $A$ is excellent, $B$ is stil excellent (of dimension $2$).
Possibly further blowing up $\tilde Y$ we can also assume that $g$ is good.

Let us fix a rank $r$  vector bundle  $\cF$ on $\tilde X$ and let us consider a filtration $\tilde \pi ^*\cF=\cF ^0\supset \cF^1 \supset ...\supset \cF ^r=0$ such that all quotients $\cL _i= \cF^i/{\cF^{i+1}}$ are line bundles. Note that such a filtration always exists, possibly after further blowing up $\tilde Y$.

Let us note that
$$(r-1) (\sum L_i)^2-2r \sum _{i<j} L_iL_j=\sum _{i<j} (L_i-L_j)^2\le 0,$$
where $L_i=c_1(g, \cL_i)$.
But $ (\sum L_i)^2= c_1(g, \tilde \pi ^*\cF)^2=\deg \tilde \pi  \cdot  c_1(f, \cF)^2$, so we have
$$\frac{\sum _{i<j} L_iL_j}{\deg \tilde \pi}\ge \frac{r-1}{2r}c_1(f, \cF)^2.$$
Therefore the following definition makes sense:

\begin{Definition}\label{definition-c_2-surfaces}
The \emph{second relative Chern class} $c_2(f, \cF)$ of $\cF$ with respect to $f$  is defined as the real number
$$c_2(f, \cF)= \inf \left(\frac{\sum _{i<j} L_iL_j}{\deg \tilde \pi }\right),$$
where the infimum is taken over all  generically finite proper morphisms $\tilde \pi: \tilde Y\to \tilde X$  from a regular surface and all filtrations  $\tilde \pi ^*\cF=\cF ^0\supset \cF^1 \supset ...\supset \cF ^r=0$, whose 
quotients $\cL _i= \cF^i/{\cF^{i+1}}$ are line bundles. 
\end{Definition}

In the rank $2$ case in the complex analytic setting this definition was introduced by  Wahl  in \cite{Wa}. His definition was generalized to arbitrary rank in \cite{La0} and studied there also in the complex analytic setting. 

The following proposition summarizes some basic properties of the relative second Chern class.

\begin{Proposition}\label{basic-properties-c_2}
\begin{enumerate}
\item For any line bundle $\cL$ on $\tilde X$ we have $c_2(f, \cF)=0$.
\item (relative Bogomolov's inequality) For any rank $r$ vector bundle $\cF$ on $\tilde X$ we have
$$\Delta (f, \cF):= 2r c_2 (f, \cF)- (r-1)c_1 (f, \cF)^2\ge 0.$$
\item For any rank $r$ vector bundle $\cF$ and any line bundle $\cL$ on $\tilde X$ we have
$$ \Delta (f, \cF \otimes \cL)= \Delta (f, \cF ).$$
\item If  $\tilde \pi: \tilde Y\to \tilde X$ is a generically finite proper morphism from a regular surface $\tilde Y$
and  $\tilde Y\xrightarrow{g} Y\xrightarrow{\pi}X$ is the Stein factorization of $f\circ \tilde \pi$ then
$$c_2(g, \tilde \pi ^*\cF) = \deg \pi \cdot c_2(f, \cF).$$
\end{enumerate}
\end{Proposition}

\begin{proof}
Properties (1)--(3) are obvious from the definition.
To show (4) let us first remark that the category of finite extensions $A\to B$ with $B$ a normal domain, is filtered. 
Given two extensions $A\to B$ and $A\to C$  we can embed the quotient fields $K(B)$ and $K(C)$
into some fixed algebraic closure of $K(A)$. Then we can find a finite field extension $K(A)\subset L$ 
that contains both  $K(B)$ and $K( C)$. Then the normalization $D$ of $A$ in $L$ gives a finite normal 
extension $A\to D$ dominating both $A\to B$ and $A\to C$.

This implies that the category of  generically finite proper morphisms from a regular surface to $\tilde X$ is cofiltered. 
Indeed, if $\tilde \pi: \tilde Y\to \tilde X$ and  $\tilde \tau: \tilde Z\to \tilde X$  are generically finite proper morphisms from regular surfaces $\tilde Y$ and $\tilde Z$ then we consider the corresponding Stein factorizations $\tilde Y\xrightarrow{g} Y\xrightarrow{\pi}X$ and $\tilde Z\xrightarrow{h} Z\xrightarrow{\tau}X$. By the above we can find a finite morphism 
$T\to X$ from a normal surface $T$, that dominates both $\pi$ and $\tau$. Then we can find a good resolution of singularities $\tilde T\to T$ with generically finite proper morphisms to both $\tilde Y$ and $\tilde Z$ (it is sufficient to find a resolution dominating main irreducible components of both $\tilde Y\times _YT$ and $\tilde Z\times _ZT$).

The above fact allows us to pull-back to $\tilde T$ any filtration of  the pullback of $\cF$ to $\tilde Z$. This gives a filtration of the pull back of $\tilde \pi^*\cF$ implying  (4).
\end{proof}

\medskip

\begin{Remark}\label{easy-inequality-for-local-ch_2}
In the following we write $\ch_2 (f, \cF) $ for $ \frac{1}{2} c_1(f, \cF)^2-c_2 (f, \cF) $. Note that definition of $c_2 (f, \cdot)$
implies that for any two vector bundles $\cF_1, \cF_2$ on $\tilde X $ we have
$$\ch_2 (f,\cF _1\oplus \cF_2)\ge \ch_2 (f,\cF _1)+\ch_2 (f,\cF _2).$$
Later we prove that if  the base field $k$ has  positive characteristic then we have equality (see Corollary \ref{a-of-sum-char-p}).  
\end{Remark}

\medskip

\begin{Remark}\label{behaviour-on-tensors}
One of the main open problems related to local relative Chern classes is their behaviour under tensor operations (see \cite[Conjecture 8.1]{La0}). For example, if we knew that one can compute $c_2(f, \Sym ^m\cF)$ using $c_1(f, \cF)$ and $c_2(f, \cF)$ using the same formulas as follow from the splitting principle for usual Chern classes, then Conjecture \ref{Wahl's-conjecture} holds.
\end{Remark}

\subsection{Local relative Riemann--Roch theorem}

\begin{Definition}
For a vector bundle $\cF$ on $\tilde X$ we define a \emph{relative Euler characteristic} $\chi (f, \cF)$ by
$$\chi (f, \cF):=\dim H^0(X, (f_*\cF)^{**}/f_*\cF) +\dim H^0(X, R^1 f_*\cF).$$  
\end{Definition}

\begin{Definition}\label{addition-a}
For any rank $r$ vector bundle $ \cF$ on $\tilde X$ we set 
$$a (f,  \cF):=\chi (f, \cF)-r\, \chi(f,\cO_{\tilde X})+
\frac{1}{2}c_1(f,  \cF)(c_1(f,  \cF)-K_{\tilde X})-c_2 (f,  \cF).$$
\end{Definition}

The proof of the following proposition is essentially the same as that of \cite[Proposition 2.9]{La0}
so, since we cannot improve upon it, we skip it.

\begin{Proposition}\label{Wahl-independence}
Let $f': \tilde X'\to X$ be a desingularization of $X$.
Let $\cF \in \Vect(\tilde X)$ and $\cF' \in \Vect (\tilde X')$ be vector bundles such that
$f_{[*]}\cF$ and $f'_{[*]}\cF'$ are isomorphic.
Then we have $a (f, \cF)=a (f',  \cF ')$.
\end{Proposition}

If $\cE$ be a reflexive coherent $\cO_X$-module then we set 
$$a(x, \cE):= a (f, \cF),$$
where $f: \tilde X\to X$ is any desingularization of $X$ and $\cF$ is any vector bundle on $\tilde X$
such that $f_{[*]}\cF\simeq \cE$. The above proposition implies that $a(x, \cE)$ is well defined. Therefore  we get a function
$a(x, \cdot ): \Refl (\cO_X) \to \RR$.
If $j: U:=X\backslash \{x\} \hookrightarrow X$ denotes the  open embedding then 
by Lemma \ref{restriction-to-open} the functor $j_*: \Vect (U )\to \Refl(\cO_X)$ is an equivalence of categories, so we can also treat $a(x, \cdot )$ as a function  $\Vect (U) \to \RR$.

\medskip

In the remaining part of the section we reprove the results of  \cite[Section 4]{La0} giving more details and providing
simpler proofs that avoid the use of reduction cycles.
First we note the following lemma.

\begin{Lemma}\label{generation-by-few-sections}
Let $\cG$ be a vector bundle of rank $r$ on $\tilde X$. If $\cG$ is globally generated outside of a finite number of $k$-points $T$ of $E$ then for general $(r+2)$ sections of $\cG$ the cokernel of the induced map
$\cO ^{\oplus (r+2)}_X\to \cG$ is supported on $T$.
\end{Lemma}

\begin{proof}
Let us choose a finite dimensional $k$-vector subspace $V\subset H^0(\tilde X , \cG)$ such that the evaluation map $V\otimes _k\cO_{\tilde X}\to \cG$ is surjective on $U:= \tilde X\backslash T$. 
We can assume that $V$ has large dimension so that $s:=\dim V-(r+2)>0$.

Let $G$ be the Grassmannian of $s$-dimensional quotient spaces of $V$. 
Then $U\times _kG\to U$ is the Grassmann scheme representing rank $s$ quotient vector bundles of the trivial bundle  $V\otimes _k \cO_U$. Let $\cK$ be the kernel of the evaluation map $V\otimes _k\cO_{\tilde X}\to \cG$ and
let us consider the subfunctor of the appropriate Grassmann functor, such that $S$-points consist of those quotients $V\otimes _k\cO_S\to \cF$ for which the induced map $\bigwedge ^{s}\cK \to \bigwedge ^{s} \cF$ vanishes. By \cite[Proposition 2.2]{Kl} this functor is represented by a closed subscheme $Z$ of  $U\times _kG$. For $x\in U(k)$, $k$-points of $Z_x\subset G$ correspond to $(r+2)$-dimensional vector subspaces $W\subset V$ such that  $\dim (W\otimes k(x)\cap \cK\otimes k(x))\ge 3$ (which is why $Z$ is also called the $3$-rd special Schubert cycle defined by $\cK$). Naively speaking, $Z$ parametrizes pairs $(x,[W\subset V])$ such that the induced map $W\otimes k(x)\to  \cG\otimes k(x)$ is not surjective. 

By \cite[Corollary 2.9]{Kl} the scheme $Z$ has relative dimension $\dim G-3$ over $U$ (this also follows from the standard dimension computation of the Schubert cell  defined by condition $\dim (W\cap \cK\otimes k(x))\ge 3$ in $G$). It follows that $\dim Z=\dim U+\dim G-3 <\dim G$. So there exists $W\subset V$ of dimension $(r+2)$  (corresponding to some point of $G\backslash p_2(Z)$, where $p_2: Z\to G$ comes from the projection $U\times _kG\to G$) such that the evaluation map $W\otimes _k\cO_{\tilde X}\to \cG$ is surjective over $U$. 
\end{proof}

\begin{Proposition}\label{c_1-full-sheaves}
Let $\cE$ be a reflexive coherent $\cO_X$-module of rank $r$ and let $\cF= f^{[*]}\cE$.
Then the following conditions are satisfied:
\begin{enumerate}
\item $f_*\cF=\cE$,
\item $\dim  H^0(X, R^1 f_*\cF)\le (r+2)\dim  H^0(X, R^1 f_*\cO _{\tilde X})$,
\item There exists some $C>0$ (independent of $\cE$) such that or every irreducible component $E_j$ of $E$ we have
$$0\le c_1(f, \cF)\cdot E_j\le  C\cdot r.$$
\end{enumerate}
\end{Proposition}

\begin{proof}
To see (1) note that we have a canonical map
$$\cE\to f_*f^*\cE\to f_*f^{[*]}\cE,$$
which is an isomorphism on $X\backslash f(E)$. Since $\cE$ is reflexive and $f_*f^{[*]}\cE$ is torsion free, 
it is an isomorphism on the whole $X$.

To see (2) let us note that $\cE$ is globally generated as $X$ is affine. So $\cF$ is globally generated outside a finite number of $k$-points (lying on $E$) and by Lemma \ref{generation-by-few-sections}
for general $(r+2)$ sections of $\cF$ the cokernel of the corresponding map
$\cO ^{\oplus (r+2)}_{\tilde X}\to \cF$ is supported on a finite set of points.
This shows that we have a surjective map $(R^1f_*\cO_{\tilde X} )^{\oplus (r+2)}\to R^1f_*\cF$, which gives (2).
Note also that for every $j$ the above maps gives a map $\cO ^{\oplus (r+2)}_{E_j}\to \cF|_{E_j}$, whose cokernel is supported on a finite set of points. Therefore $c_1(f, \cF)\cdot E_j=\deg \cF|_{E_j} \ge 0$.

Now let us consider a short exact sequence 
$$0\to \cF\to \cF(E_j)\to \cF(E_j)|_{E_j}\to 0.$$
(1) implies that $f_* \cF=\cE=f_*(\cF(E_j))$, so the map $ f_*(\cF(E_j)|_{E_j})\to R^1f_*\cF$ is injective.
Therefore we have
$$ \chi ( \cF(E_j)|_{E_j})\le \dim H^0(\tilde X, \cF(E_j)|_{E_j})\le \dim  H^0(X, R^1 f_*\cF)\le (r+2)\dim  H^0(X, R^1 f_*\cO _{\tilde X}).$$
But by the Riemann--Roch theorem on $E_j$ we have
$ \chi ( \cF(E_j)|_{E_j})= r\chi (\cO_{E_j} ) +r E_j^2+ \deg \cF|_{E_j}$, so we get
$$c_1(f, \cF)\cdot E_j \le (r+2)\dim  H^0(X, R^1 f_*\cO _{\tilde X}) - r\chi (\cO_{E_j} ) -r E_j^2,$$
which gives the second inequality in (3).
\end{proof}

\medskip

\medskip

\begin{Corollary}\label{corr-c_1-full-sheaves}
	In the notation of Proposition \ref{c_1-full-sheaves} the $\QQ$-divisor $-c_1(f, \cF)$ is effective and there exists a constant $\tilde C>0$ depending only on $f: \tilde X\to X$ such that for every $\cE$ we have
	$$-c_1(f, \cF)^2\le \tilde C\cdot r^2.$$
\end{Corollary}

\begin{proof}
	The first assertion follows from the inequalities $0\le c_1(f, \cF)\cdot E_j$ (see \cite[(7)]{Gi}). To see the second assertion let us write  $c_1(f, \cF)=-\sum \alpha_i E_i$ for some non-negative rational numbers $\alpha_i$. Then by Cramer's rule the $\alpha_i$ depend linearly on the numbers $c_1(f, \cF)\cdot E_j$. So the assertion follows from the last part of Proposition \ref{c_1-full-sheaves}.
\end{proof}

\medskip

\begin{Lemma}\label{c_2-global-gen}
If $\cG$ is a globally generated rank $r$ vector bundle on $\tilde X$ then $c_2(f, \cG)\le 0$.
\end{Lemma}

\begin{proof}
Let us choose a finite dimensional $k$-vector subspace $V\subset H^0(\tilde X , \cG)$ such that the evaluation map $V\otimes _k\cO_{\tilde X}\to \cG$ is surjective.
Let $G$ be the Grassmannian of $(r-1)$-dimensional $k$-subspaces of $V$ and let $Z\subset \tilde X\times G$ be the subscheme parametrizing pairs $(x,[W\subset V])$ such that the induced map $W\otimes k(x)\to \cG\otimes k(x)$ is not injective. Strictly speaking, we should again use the functorial approach of \cite{Kl} (as in the proof of Lemma \ref{generation-by-few-sections}) but we sketch a naive approach leaving formalization of proof to the reader.

 Let us consider projections $p_1: Z\to \tilde X$ and $p_2: Z\to G$.
For any $k$-point of $\tilde X$ the map $ V\otimes k(x)\to \cG\otimes k(x) $ is surjective. So if $K$ denotes its kernel (which is of dimension $\dim V-r$) then we have
$$p_1^{-1}(x)\simeq \{ [W\subset V]\in G: \dim ((W\otimes k(x))\cap K)\ge 1\} .$$
A standard computation shows that this Schubert cell has 
codimension $2$ in $G$. It follows that $\dim 
p_1^{-1} (E)=\dim E+\dim G-2 <\dim G$. So there exists $W\subset V$ of dimension $(r-1)$ 
(corresponding to some point of $G\backslash p_2(p_1^{-1} (E))$) such that the evaluation map $W\otimes _k\cO_{\tilde X}\to \cG$ is injective and its cokernel is locally free along $E$. 
This implies that we have a short exact sequence
$$0\to \cO_{\tilde X} ^{\oplus (r-1)}\to \cG \to \cL \to 0,$$
where $\cL$ is a line bundle.
This immediately implies the required inequality.
\end{proof}

\medskip

The following theorem is \cite[Corollary 4.13]{La0} with a different proof that avoids the use of reduction cycles.

\begin{Theorem}\label{bound-on-a(x,E)}
	There exists some constants $A$ and $B$ depending only on $X$ such that for every 
	reflexive coherent $\cO_X$-module  $\cE$ of rank $r$ we have
	$$A r^2\le  a(x, \cE) \le Br.$$
\end{Theorem}

\begin{proof}
	Let $\cE$ be a reflexive coherent $\cO_X$-module of rank $r$ and let $\cF= f^{[*]}\cE$.
	Let us fix an $f$-very ample line bundle $\cO_{\tilde X} (1)$.
	By Serre's theorem there exists some $n_0$ such that for all $n\ge n_0$ such that for all irreducible components $E_j$ of $E$ we have $R^1f_*(\cO_{\tilde X} (-E_j)\otimes \cO_{\tilde X} (n))=0$.
	
	Let us recall that by Lemma \ref{generation-by-few-sections}
	the cokernel of the map $\cO ^{\oplus (r+2)}_{\tilde X}\to \cF$ determined by  $(r+2)$ general sections of $\cF$
	is supported on a finite set of points. Twisting it by $\cO_{\tilde X} (-E_j)\otimes \cO_{\tilde X} (n)$ we see that 
	$$R^1f_*(\cF (n)\otimes \cO_{\tilde X} (-E_j))=0$$
	for all $n\ge n_0$.
	So we have surjective maps $$H^0(\tilde X , \cF (n))\to H^0( E_j, \cF (n)|_{E_j}).$$
	This implies that for all $n\ge n_0$ the bundle $\cF (n)$ is globally generated.
	So by Lemma \ref{c_2-global-gen} we have $c_2(f, \cF (n_0))\le 0$ and hence
	\begin{align*}
		0\le \Delta (f, \cF)&=\Delta(f, \cF (n_0))\le - (r-1)c_1(f, \cF (n_0))^2\\
		&=
		- (r-1)\left( c_1(f, \cF )^2 +n_0 c_1(f, \cF ).c_1(f, \cO_X(1))+n_0^2 c_1(f, \cO_X(1))^2 \right) .
	\end{align*}
	By Proposition \ref{c_1-full-sheaves} we also know that $0\le \chi (f, \cF)\le  (r+2)  \chi (f, \cO_{\tilde X})$. Since
	$$a (x,  \cE)=\chi (f, \cF)-r\, \chi(f,\cO_{\tilde X})+
	\frac{1}{2r}c_1(f,  \cF)(c_1(f,  \cF)-rK_{\tilde X})-\frac{\Delta (f,  \cF)}{2r},$$
	the required inequalities follow from Proposition \ref{c_1-full-sheaves}  and Corollary \ref{corr-c_1-full-sheaves}.
\end{proof}

\begin{Remark}
	The proof shows that in the above theorem one can find $A$ and $B$ that depend only on  numerical invariants of $X$ and its fixed resolution $f: \tilde X\to X$. More precisely,  these constants can be determined by $\dim R^1f_* \cO_X $ (called the geometric genus of the singularity $X$), discrepancies of the exceptional divisor and the intersection matrix of exceptional curves of $f$. This fact is useful when studying how Chern classes change in families of reflexive sheaves on normal surfaces.
\end{Remark}

\begin{Remark}
	It is natural to expect that in notation of  the proof of Theorem \ref{bound-on-a(x,E)},
	there exists a constant $\tilde A$ such that $\Delta (f, \cF)\le \tilde A r^2$.  This would imply that one can find $A$ such that $A r^2\le  a(x, \cE)$. This conjecture is equivalent to  \cite[Conjecture 8.2]{La0}.
\end{Remark}

\subsection{Characterization of the relative second Chern class in positive characteristic}

In this subsection we assume that the base field $k$ has characteristic $p>0$.

\begin{Theorem}
There exists a uniquely determined function $c_2(f, \cdot ): \Vect (\tilde X) \to \RR$, $\cF\to c_2 (f,\cF)$ such that  the following conditions are satisfied:
\begin{enumerate}
\item For every $\cF\in  \Vect (\tilde X)  $ we have
 $c_2 (f, F_{\tilde X}^{*} \cF)=p^{2} c_2 (f, \cF ).$
\item There exists a function $\varphi: \NN\to \RR _{\ge 0}$  such that
for every rank $r$ vector bundle $\cF\in  \Vect (\tilde X)  $ we have
$$\left|\chi (f, \cF) +\frac{1}{2}c_1(f,  \cF)(c_1(f,  \cF)-K_{\tilde X})-c_2 (f,  \cF) -
r\, \chi(f,\cO_{\tilde X})\right|\le \varphi  (r) .
$$
\end{enumerate}
\end{Theorem}

\begin{proof}
Existence of the function $c_2(f, \cdot )$ was proven in previous sections. More precisely, (1)
follows from the last part of Proposition \ref{basic-properties-c_2} applied to the Frobenius morphism.
Condition (2) follows from Theorem \ref{bound-on-a(x,E)}. 

To prove uniqueness note that $c_1(f, (F_X^{m})^{*} \cF) =p^m c_1(f, \cF)$. So  (1) and (2) imply that
\begin{align*}
&\left| \chi (f, (F_X^{m})^{*} \cF)  +\frac {1}{2}p^mc_1 (f,\cF). (p^mc_1(f,\cF) -K_{\tilde X})-p^{2m} c_2 (f,\cF) -r\, 
\chi (f, \cO_{\tilde X}) \right| \\
&\le \varphi (r)
\end{align*}
Dividing the above inequality by $p^{2m}$ and passing to the limit we get
$$c_2 (f, \cF) = \frac{1}{2} c_1(f, \cF)^2+\lim _{m\to \infty} \frac{\chi (f, (F_X^{m})^{*} \cF )}{p^{2m}}.$$
\end{proof}

\begin{Remark}
	Theorem  \ref{bound-on-a(x,E)} shows that in fact we can take $\varphi$ to be quadratic and independent of $f$. This is not needed for the proof of uniqueness of $c_2(f, \cdot )$.
\end{Remark}

\medskip

The last formula in the above proof and additivity of the relative Euler characteristic imply the following corollary.

\begin{Corollary}\label{a-of-sum-char-p}
For any $\cF_1, \cF_2\in  \Vect (\tilde X)  $ we have
$$\ch_2 (f,\cF _1\oplus \cF_2)=\ch_2 (f,\cF _1)+\ch_2 (f,\cF _2).$$
Moreover, for any coherent reflexive $\cO_X$-modules $\cE_1$ and $\cE_2$ we have
$$a (x,\cE _1\oplus \cE_2)=a (x,\cE _1)+a (x,\cE _2).$$
\end{Corollary}

\subsection{Conjectural characterization of the relative second Chern class in characteristic zero}

Assume that the base field $k$ has characteristic zero. In this case one expects that the following asymptotic Riemann--Roch formula works (see the conjecture in \cite[Introduction]{Wa} for the rank $2$ case;
we need this conjecture in an arbitrary rank as hinted in \cite[Section 8]{La0}).

\begin{Conjecture}\label{Wahl's-conjecture}
If $\cF$ is a rank $r$ vector bundle on $\tilde X$ then
$$\chi (f, \Sym ^m \cF)=-\frac{m^{r+1}}{r!} \left(c_1(f, \cF)^2-c_2(f, \cF) \right) + O(m^r).$$
\end{Conjecture}

This conjecture would allow us to characterize $c_2(f, \cF)$ as
$$c_2 (f, \cF) =  c_1(f, \cF)^2+r! \lim _{m\to \infty} \frac{\chi (f, \Sym ^m \cF  )}{m^{r+1}}.$$
Using Theorem \ref{bound-on-a(x,E)} it is easy to prove inequality $\le$ in the above conjecture
(cf. \cite[Theorem 4.15]{La0} for the rank $2$ case).
This allows us to consider 
$$\liminf _{m\to \infty} \frac{\chi (f, \Sym ^m \cF  )}{m^{r+1}}$$
(or similar limits) to define local relative Chern classes (see, e.g., \cite{La2} for one example of use of such definition).

\medskip

Let us recall that a quotient surface singularity is a quotient of  the spectrum of a regular $2$-dimensional ring
by a linear action of a finite group. The following result follows from \cite[Theorem 5.1]{La0}.

\begin{Theorem}\label{Wahl's-conj-for-quotients}
Assume that $X=\Spec A$, where $A$ is the henselization of a local ring of a quotient surface singularity.  
Then Conjecture \ref{Wahl's-conjecture} holds for any desingularization of $X$. Moreover, for any 
vector bundle $\cF$ on $\tilde X$ the number $c_2(f, \cF)$ is rational.
\end{Theorem}

See \cite{La0} for further discussion of Conjecture \ref{Wahl's-conjecture} 
and its proof in some other cases.

\section{Chern classes of reflexive sheaves on normal surfaces}

Let $X$ be a proper normal variety defined over an algebraically closed field $k$.
For any coherent $\cO_X$-module $\cE$ of rank $r$ the sheaf $\det \cE =(\bigwedge ^r \cE)^{**}$ is 
a reflexive coherent $\cO_X$-module of rank $1$. So we can define $c_1(\cE)\in A^1(X)$ 
as the class of a divisor $D$ such that $\det \cE =\cO_X (D)$.  

In this section we assume that $\dim X=2$.

\subsection{Second Chern class}\label{pullback}

Let  $f: \tilde X\to X$ be any resolution of singularities and 
let $E$ be its exceptional locus. For every $x\in f(E)$ we consider the map $\nu_x: \Spec \cO_{X,x}^h\to X$ from 
the spectrum of the henselization of the local ring of $X$ at $x$
and the base change $f_x=\nu_x^*f: \tilde X _x\to \Spec \cO_{X,x}^h$ of $f$ via $\nu_x$.
Note that the group of divisors on $\tilde X _x$ that are supported on
 the exceptional locus  $E_x$ of $f_x$ embeds into $A_1(\tilde X)$. For any vector bundle 
$\cF$ on $\tilde X$ we  write $c_1 (f_x, \cF)\in A_1(\tilde X)\otimes \QQ$ for the 
image  of $c_1 (f_x, \cF |_{ \tilde X _x})$.
Then for any $\cF\in \Vect (\tilde X)$  we have
$$c_1 (\cF)=f^*c_1 (f_{[*]} \cF)+\sum _{x\in f(E)}c_1 (f_x, \cF)$$
in $ A_1(\tilde X)\otimes \QQ$. So if $\cE\in \Refl(\cO_X) $ and  we choose some vector bundle $\cF$ on $\tilde X$ such that $f_{[*]}\cF\simeq \cE$ then
$$f^*c_1 (\cE)=c_1 (\cF)-\sum _{x\in f(E)} c_1 (f_x, \cF) $$ 
treated as an element of $ A_1(\tilde X)\otimes \QQ.$

\medskip

\begin{Definition}
If $\cE\in \Refl(\cO_X) $ then we choose some vector bundle $\cF$ on $\tilde X$ such that $f_{[*]}\cF\simeq \cE$.
Then we use  the homomorphism $f_*: A_0(\tilde X)\to A_0(X)$ to define the \emph{second Chern class} of $\cE$ as
$$c_2 (\cE):=f_*c_2 (\cF)-\sum _{x\in f(E)} c_2 (f_x, \cF) \, [x]$$ 
treated as an element of $ A_0(X)\otimes \RR.$
\end{Definition}

Note that  the function $c_2: \Refl(\cO_X) \to A_0(X)\otimes \RR$ is well  defined, i.e., the class $f_*c_2 (\cF)-\sum _{x\in f(E)} c_2 (f_x, \cF) \, [x]$  does not depend on the choice of $\cF$ and $f$.
Namely, if we choose   another resolution of singularities $f': \tilde X'\to X$ and $\cF'\in \Vect (\tilde X')$ such that  $ f'_{[*]}\cF'\simeq \cE$ then $f_*c_2 (\cF ')-f'_*c_2 (\cF')$ is a well-defined $0$-cycle supported on $S=f(E)\cup f(E')$ and Proposition \ref{Wahl-independence} implies equality of the corresponding degrees locally at each point of $S$.

\medskip

Using the above definition we can also define for $\cE\in \Refl(\cO_X) $  the \emph{discriminant} $\Delta (\cE)$ as
$$\Delta (\cE):= 2r c_2 (\cE)- (r-1)c_1 (\cE)^2 \in  A_0(X)\otimes \RR,$$
where $r$ is the rank of $\cE$. Let us also recall that we have the degree map $\int _X: A_0(X)\otimes \RR \to \RR$.

The following proposition summarizes some of the basic properties of Chern classes on normal surfaces:

\begin{Proposition}\label{global-basic-properties-c_2}
\begin{enumerate}
\item For any rank $1$ reflexive coherent $\cO_X$-module $\cL$ on $X$ we have $c_2(\cL)=0$.
\item For any $\cE\in \Refl(\cO_X) $  and any rank $1$ reflexive coherent $\cO_X$-module $\cL$ on $X$ we have
$ \Delta (\cE \hat \otimes \cL)=\Delta (\cE).$
\item For any vector bundle $\cF$ on $\tilde X$ we have
$\int_{\tilde X} \Delta ( \cF) \ge  \int_{X} \Delta ( f_{[*]}\cF).$\item If  $ \pi: Y\to X$ is a finite morphism from a normal  surface $Y$ then for any $\cE\in \Refl(\cO_X) $ we have
$\int_Yc_2 ( \pi ^{[*]}\cE) = \deg \pi \cdot \int_Xc_2(\cE).$
\end{enumerate}
\end{Proposition}

\begin{proof}
All the above properties follow easily from  the corresponding properties of relative local Chern classes listed in Proposition \ref{basic-properties-c_2} and from properties of Chern classes of vector bundles on smooth surfaces. 
\end{proof}

\begin{Remark}\label{c_2-direct-sum}
Note that for any two Weil divisors $D_1$ and $D_2$ on $X$ we have
$$\int_X c_2 (\cO _X (D_1) \oplus \cO_X (D_2))= D_1.D_2.$$
In the local case this follows from \cite[Proposition 2.5]{Wa}. The global assertion follows from this fact and definition
of Mumford's intersection numbers on normal surfaces. Note that  this equality implies that Mumford's intersection numbers behaves well under finite coverings (see Proposition \ref{global-basic-properties-c_2}, (4)).
\end{Remark}

\subsection{Riemann--Roch theorem on normal surfaces}

For $\cE\in  \Refl(\cO_X)$ we write $a(x, \cE)$ for $ a(x, \nu_x^* \cE)$ (see the previous subsection for the notation).
Then we have the following Riemann--Roch type theorem:

\begin{Theorem}\label{RR-theorem-easy}
Let $X$ be a normal proper algebraic surface defined over an algebraically closed field $k$.
Then for any  $\cE\in \Refl(\cO_X)$
we have 
$$\chi (X, \cE) = \frac {1}{2}c_1 (\cE). (c_1(\cE) -K_X)- \int_Xc_2 (\cE) +\rk \cE\cdot \chi (X, \cO_X)+\sum _{x\in \Sing X} a(x, \cE).$$ 
\end{Theorem}

\begin{proof}
An easy computation using the Leray spectral sequence shows that for any vector bundle $\cF$ 
on $\tilde X$ we have
$$\chi (\tilde X, f_{[*]}\cF)=\chi (X, \cF)+\sum _{x\in f(E)} \chi (f_x, \cF).$$
For any  $\cE\in \Refl(\cO_X)$ we choose $\cF$ such that $\cE\simeq f_{[*]}\cF$ (e.g., one can take $\cF=f^{[*]}\cE$).
Then the required formula follows from the Riemann--Roch theorem for $\cF$ on $\tilde X$, 
Definition \ref{addition-a} and definitions of $c_1(\cE)$ and $c_2(\cE)$.
\end{proof}

\subsection{Second Chern class in positive characteristic}

The following theorem shows that in positive characteristic  
the second Chern class is uniquely determined by two very simple properties.

\begin{Theorem}\label{RR-theorem}
Let $X$ be a normal proper algebraic surface defined over an algebraically closed field of characteristic $p>0$.
Then there exists a uniquely determined function $\int _Xc_2: \Refl(\cO_X) \to \RR$, $\cE\to \int_Xc_2 (\cE)$, such that  the following conditions are satisfied:
\begin{enumerate}
\item  There exists a function $\varphi: \NN\to \RR_{\ge 0}$ such that for every $\cE\in \Refl(\cO_X)$ of rank $r$ we have
$$\left|\chi (X, \cE) -\left(\frac {1}{2}c_1 (\cE). (c_1(\cE) -K_X)- \int_Xc_2 (\cE) +r\chi (X, \cO_X)\right)\right|\le \varphi(r).$$
\item For every $\cE\in \Refl(\cO_X)$  we have $\int _X c_2 (F_X^{[m]} \cE )=p^{2m} \int_Xc_2 (\cE).$
\end{enumerate}
\end{Theorem}

\begin{proof}
Existence of $\int _Xc_2$ satisfying the above conditions was already proven. Namely, by Theorem \ref{RR-theorem-easy} we have
$$\chi (X, \cE) -\left(\frac {1}{2}c_1 (\cE). (c_1(\cE) -K_X)- \int_Xc_2 (\cE) +r\chi (X, \cO_X)\right) =a(\cE),$$
where $a(\cE):=\sum _{x\in \Sing X} a(x, \cE)$. Since $X$ has finitely many singular points,
(1) follows from Theorem \ref{bound-on-a(x,E)} (and in fact $\varphi$ can be taken linear in $r$).  The second condition follows from Proposition \ref{global-basic-properties-c_2}, (4).

Now let us prove that $\int _Xc_2$ is uniquely determined by the conditions (1) and (2).
Note that $c_1(F_X^{[m]} \cE) =p^m c_1(\cE)$. So  using (1) we get
\begin{equation*}
\left | \chi (X, F_X^{[m]} \cE) - \left(\frac {1}{2}p^mc_1 (\cE). (p^mc_1(\cE) -K_X)- p^{2m} \int_Xc_2 (\cE) +r \chi (X, \cO_X)\right)
\right| \le \varphi(r).
\end{equation*}
Dividing the above inequality by $p^{2m}$ and passing to the limit we get
$$\int_Xc_2 (\cE) = \frac{1}{2} c_1(\cE)^2- \lim _{m\to \infty} \frac{\chi (X, F_X^{[m]} \cE) }{p^{2m}}.$$
\end{proof}

\medskip

\begin{Remark}\label{ch_2-on-surfaces}
We can also talk about the second  Chern character $ \ch _2 (\cE)$, which is defined as usual by setting $\ch _2 (\cE)= \frac{1}{2} c_1(\cE)^2- c_2 (\cE)$. By the above, we see that 
$$\int_X \ch _2 (\cE)=\lim _{m\to \infty} \frac{\chi (X, F_X^{[m]} \cE) }{p^{2m}}.$$
Existence of this limit is the main reason why we can define the second  Chern character for higher dimensional varieties in positive characteristic (see Theorem \ref{properties-of-ch_2}).
\end{Remark}

\subsection{Second Chern class in characteristic zero}

Let us assume that the base field $k$ has characteristic $0$. 
The following result follows from the Riemann--Roch theorem on resolution of singularities of $X$ and from Theorem \ref{Wahl's-conj-for-quotients}.

\begin{Theorem}\label{asymptotic-RR-quotient-surfaces}
Let $X$ be a normal proper algebraic surface with at most quotient singularities. Then for any
$\cE\in \Refl(\cO_X)$  of rank $r$ we have  $\int_X c_2(\cE) \in \QQ$ and
$$\chi (X, \Sym ^{[m]}\cE)=\frac{m^{r+1}}{r!} \left(c_1(\cE)^2-\int_X c_2(\cE) \right) + O(m^r). $$
\end{Theorem}

\section{Chern classes in higher dimensions}

In this section we develop the theory of  the second Chern classes and characters for
reflexive sheaves on higher dimensional normal varieties. This works well in case of positive characteristic. 
In characteristic zero in dimensions $>2$ the theory depends heavily on  Conjecture \ref{Wahl's-conjecture}, so we can do it only for varieties with at most quotient singularities in codimension $2$.

\medskip

In this section we assume that $n\ge 2$.

\subsection{The second Chern character in positive characteristic}

Let us assume that the base field $k$ has positive characteristic $p$.

\begin{Proposition}\label{existence-of-limit}
Let $\cE$ be a reflexive coherent $\cO_X$-module and let $L_1,..., L_{n-2}$ be line bundles on $X$.
Then 
the sequence $$\left(\frac {  \chi (X, c_1(L_1)...c_1(L_{n-2}) \cdot {F_X^{[m]}\cE} ) }{p^{2m}} \right)_{m\in \NN}$$
 is convergent to some real number. 
 \end{Proposition}

\begin{proof}
First we use similar arguments to that from the proof of Proposition \ref{limit-square-of-divisor} to reduce to the case, when all $L_i$ are very ample.
For any line bundle $M_1$ on $X$, Lemma \ref{Kollar}, (1) and Lemma \ref{reflexive-in-K(X)} imply that
$$c_1(M_1)c_1(L_1)...c_1(L_{n-2}) \cdot {F_X^{[m]}\cE} =c_1(M_1)c_1(L_1)...c_1(L_{n-2}) \cdot 
(\cO_X^{\oplus r}-[\cO_{ -c_1(F_X^{[m]}\cE )}] ) ,$$
 where $r$ denotes the rank of $\cE$.
Since $c_1(F_X^{[m]}\cE ) =p^mc_1 (\cE)$, Lemma \ref{difference} shows that $[\cO_{ -c_1(F_X^{[m]}\cE )}] +p^m
[\cO_{c_1 (\cE)}]\in K_{n-2}(X)$. So Lemma  \ref{Kollar}, (1) gives
$$c_1(M_1)c_1(L_1)...c_1(L_{n-2}) \cdot  [\cO_{ -c_1(F_X^{[m]}\cE )}] =p^mc_1(M_1)c_1(L_1)...c_1(L_{n-2}) \cdot  [\cO_{ c_1(\cE )}] $$
and  we get
\begin{align*}
\lim _{m\to \infty}\frac{\chi (X,  c_1(M_1)c_1(L_1)...c_1(L_{n-2}) \cdot {F_X^{[m]}\cE} )}{p^{2m}}
=0.
\end{align*}
But by Lemma \ref{Kollar} (3) we have
\begin{align*}
&{  \chi (X, c_1(L_1\otimes M_1)c_1(L_2)...c_1(L_{n-2}) \cdot  {F_X^{[m]}\cE}  ) }
={  \chi (X, c_1(L_1)...c_1(L_{n-2}) \cdot {F_X^{[m]}\cE}   ) }\\
&+ \chi (X, c_1(M_1)c_1(L_2)...c_1(L_{n-2}) \cdot {F_X^{[m]}\cE}  ) - 
 {  \chi (X, c_1(M_1)c_1(L_1)...c_1(L_{n-2}) \cdot {F_X^{[m]}\cE}  ) }.
\end{align*}
So if  the limits $$\lim _{m\to \infty }\frac {  \chi (X, c_1(L_1)...c_1(L_{n-2}) \cdot {F_X^{[m]}\cE}  ) }{p^{2m}}$$ and 
$$\lim _{m\to \infty }\frac {  \chi (X, c_1(M_1)c_1(L_2)...c_1(L_{n-2}) \cdot {F_X^{[m]}\cE}  ) }{p^{2m}}$$
exist, then $$\lim _{m\to \infty }\frac {  \chi (X, c_1(L_1\otimes M_1)c_1(L_2)...c_1(L_{n-2}) \cdot {F_X^{[m]}\cE} ) }{p^{2m}}$$
also exists and it is equal to their sum.

Since any line bundle $L$ can be written as $A\otimes B^{-1}$ for some very ample line bundles $A$ and $B$ and the formula in the sequence is symmetric in $(L_1, ..., L_{n-2})$, it is sufficient to prove convergence of the sequence assuming that all line bundles $L_i$ are very ample.

Let $K$ be any uncountable algebraically closed field containing $k$ and let $X_K\to X$ be the base change. Since
$$ \chi (X_K, c_1((L_1)_K)...c_1((L_{n-2})_K) \cdot {F_{X_K}^{[m]}\cE_K} )= \chi (X, c_1(L_1)...c_1(L_{n-2}) \cdot {F_X^{[m]}\cE} ),$$ it is sufficient to prove convergence of the considered sequence after base change to $K$. So in the following we can assume that our base field $k$ is uncountable.

By Theorem \ref{Bertini} a general divisor  $H_1\in |\cL_1|$ is normal and irreducible. Since  $F_{X}^{[m]} \cE $ is  reflexive, $H_1$ is  $F_{X}^{[m]} \cE $-regular
and the restriction $(F_{X}^{[m]} \cE )|_{H_1}$ is torsion free as $\cO_{H_1}$-module (see, e.g., \cite[Lemma 1.1.13]{HL}). By \cite[Corollary 1.1.14]{HL}, for fixed $m$ and general $H_1$ the restriction $(F_{X}^{[m]} \cE )|_{H_1}$ is also reflexive. Therefore, since  $k$ is uncountable, there exists a divisor $H_1\in |\cL_1|$ 
such that $H_1$ is normal, irreducible and the restriction $(F_{X}^{[m]} \cE )|_{H_1}$ is reflexive for all non-negative integers $m$. Since $\cE $ is locally free outside of a closed subset of codimension $\ge 2$ in $X$, there exists a point $x\in H_1$ such that $\cE _x$ is a free $\cO_{X, x}$-module of some rank $r$. Therefore $\cE |_{H_1}$ has the same rank $r$ as $\cE $
and  we have $(F_{X}^{[m]} \cE )|_{H_1}=  F_{H_1}^{[m]} (\cE |_{H_1}).$ 
In particular, we have $c_1(L_1)\cdot F_{X}^{[m]} \cE= F_{H_1}^{[m]} (\cE |_{H_1})$ in $K(X)$.

Proceeding in the same way, we can construct a sequence of divisors $H_i\in  |\cL_i|$, $i=1,...,n-2$, such that for all  non-negative integers $m$ we have
\begin{enumerate}
\item the intersection  $X_i:=\bigcap _{j\le i} H_j $ is normal and irreducible,
\item the restriction $(F_{X}^{[m]} \cE )|_{X_i }$ is reflexive of rank $r$,
\item we have  $F_{X_i}^{[m]} (\cE |_{X_i })=
c_1(L_1)...c_1(L_i)\cdot F_{X}^{[m]} \cE.$
\end{enumerate}

 In particular,  $S=X_{n-2} $ is a normal surface and by   the Riemann--Roch theorem on $S$ (see Remark \ref{ch_2-on-surfaces}) we obtain
$$\lim _{m\to \infty }\frac {  \chi (X, c_1(L_1)...c_1(L_{n-2}) \cdot {F_X^{[m]}\cE}  ) }{p^{2m}}=\lim _{m\to \infty}\frac{\chi(S, F_{S}^{[m]} (\cE |_S )) }{p^{2m}}= \int _S \ch _2 (\cE |_S).$$
 \end{proof}

\medskip

\begin{Theorem}\label{properties-of-ch_2}
Let us fix $\cE\in \Refl(\cO _X)$ and
consider the map $\int_X\ch_2 : N^1( X) ^{\times (n-2)}\to \RR$ sending
$(L_1,..., L_{n-2})$ to 
$$\int _X \ch _2 (\cE)L_1...L_{n-2}:=
\lim _{m\to \infty }\frac {  \chi (X, c_1(L_1)...c_1(L_{n-2}) \cdot {F_X^{[m]}\cE}  ) }{p^{2m}}.$$
This map satisfies the following properties:
\begin{enumerate}
\item It is $\ZZ$-linear in all $L_i$.
\item It is symmetric in $L_1,...,L_{n-2}$.
\item If $\cE$ is a vector bundle on $X$ then $$\int _X \ch _2 (\cE)L_1...L_{n-2}=\int_X \ch _2 (\cE)\cap c_1(L_1)\cap ...\cap c_1(L_{n-2}) \cap [X].$$
\item If $k\subset K$ is an algebraically closed field extension then 
$$\int _{X_K} \ch _2 (\cE_K)(L_1)_K...(L_{n-2})_K=\int _X \ch _2 (\cE)L_1...L_{n-2}.$$
\item If $n>2$ and  $ L_1$ is very ample then for a very general hypersurface $H\in |L_1|$ we have
$$\int _X \ch _2 (\cE)L_1...L_{n-2}=\int _{H}\ch_2 (\cE|_{H})L_2|_{H}...L_{n-2}|_{H}.$$
\end{enumerate}
\end{Theorem}

\begin{proof}
The map is well defined by Proposition \ref{existence-of-limit}. (1) follows from the first part of the proof of this proposition. (2) follows from the definition and Lemma \ref{Kollar}, (2). (4) follows from the definition and the fact that Euler characteristic does not change under base field extension.
To prove (5) let us first note that by Bertini's theorem  general $H\in |L_1|$ is normal and irreducible. So both sides of the equality are well defined. Moreover, by \cite[Corollary 1.1.14]{HL} 
for fixed $m$ and general $H\in |L_1|$, the restriction $(F_X^{[m]}\cE)|_{H}$ is reflexive and hence isomorphic to ${F_{H}^{[m]}(\cE |_{H})}$ (if $\cE|_H$ is also reflexive). So for very general $H\in |L_1|$ we have
$$\chi (X, c_1(L_1)...c_1(L_{n-2}) \cdot {F_X^{[m]}\cE}  )  =\chi (H, c_1(L_2|_{H})...c_1(L_{n-2}|_{H}) \cdot {F_{H}^{[m]}(\cE |_{H})}  ) $$
for all $m$ at the same time.  Dividing the above equality by $p^{2m}$ and passing to the limit gives (5).
Finally, (3) follows from (4), (5) and the analogous fact in the surface case.
\end{proof}

\medskip

\begin{Definition}
For any reflexive coherent $\cO_X$-module $\cE$ of rank $r$ and any line bundles  $ L_1,..., L_{n-2}$ 
we set:
$$\int_X c_1 ^2(\cE) L_1...L_{n-2} :=  c_1 (\cE)^2.L_1...L_{n-2},$$
$$\int _X c_2 (\cE)L_1...L_{n-2}:= \frac{1}{2} \int_X c_1 ^2(\cE) L_1...L_{n-2}-\int _X \ch _2 (\cE)L_1...L_{n-2},$$
$$\int _X \Delta (\cE)L_1...L_{n-2}:=2r \int _X c_2 (\cE) L_1...L_{n-2} -(r-1)\int_X c_1 ^2(\cE) L_1...L_{n-2}.$$
\end{Definition}

\medskip

If $D$ is a Weil divisor on $X$ then $\cO_X(D)$ is a reflexive $\cO_X$-module of rank $1$. In this case our definitions agree and since $F_X^{[m]} \cO_X (D)= \cO_X (p^mD)$, we have
\begin{align*}
\int _X \ch _2 (\cO_X(D))L_1...L_{n-2} &= \lim _{m\to \infty }\frac {  \chi (X, c_1(L_1)...c_1(L_{n-2}) \cdot  \cO_X (p^mD) ) }{p^{2m}}\\
&= \lim _{m\to \infty }\frac {  \chi (X, c_1(L_1)...c_1(L_{n-2}) \cdot  [\cO_{p^mD} ])}{(p^{m})^2} \\
&= \frac{1}{2}D^2.L_1...L_{n-2}=\frac{1}{2}\int_X c_1 ^2(\cO_X(D)) L_1...L_{n-2}.
\end{align*}
In particular,  $\int _X c_2 (\cO_X(D))L_1...L_{n-2}=0$ and $\int _X \Delta (\cO_X(D))L_1...L_{n-2}=0$.

\medskip

\subsection{Properties of Chern classes in positive characteristic}

In this subsection we  assume that $k$ has positive characteristic. Apart from Theorem \ref{properties-of-ch_2}, the second Chern classes have the following properties analogous to that from Proposition \ref{global-basic-properties-c_2}.  We can assume in this proposition that $n>2$.

\begin{Proposition}\label{further-properties-char-p}
For any  line bundles $L_1,..., L_{n-2}$ on $X$ the following conditions are satisfied:
\begin{enumerate}
\item For any rank $1$ reflexive coherent $\cO_X$-module $\cL$ on $X$ we have $$\int _X c_2 (\cL)L_1...L_{n-2}=0.$$
\item For any $\cE\in \Refl(\cO_X) $  and any rank $1$ reflexive coherent $\cO_X$-module $\cL$ on $X$ we have
$$ \int_X \Delta (\cE \hat \otimes \cL)L_1...L_{n-2}=\int_X \Delta (\cE) L_1...L_{n-2}.$$
\item If $f: \tilde X\to X$ is a resolution of singularities and it has separably generated residue field extensions
then for any vector bundle $\cF$ on $\tilde X$ we have
$$\int_{\tilde X} \Delta ( \cF)f^*L_1...f^*L_{n-2} \ge  \int_{X} \Delta ( f_{[*]}\cF) L_1...L_{n-2}. $$
\item Let  $ \pi: Y\to X$ be a finite morphism from a normal  projective variety $Y$. Assume that either $\pi$ has separably generated residue field extensions or it is  the Frobenius morphism (or composition of such maps).
Then for any $\cE\in \Refl(\cO_X) $ we have
$$\int_Yc_2 ( \pi ^{[*]}\cE) \pi^*L_1...\pi^*L_{n-2}= \deg \pi \cdot \int_Xc_2(\cE)L_1...L_{n-2}.$$
\end{enumerate}
\end{Proposition}

\begin{proof}
Let us first assume that $k$ has positive characteristic. Passing to the base change we can assume that the base field is uncountable. By Theorem \ref{properties-of-ch_2}, (4) it is sufficient to check  properties (1) and (2)
after restricting to a very general complete intersection surface. In this case these properties follow from Proposition \ref{global-basic-properties-c_2}, (1) and (2). If $\pi$ is the Frobenius morphism then (4) follows from the definition of  
$\int_Xc_2(\cE)L_1...L_{n-2}.$ Apart from that case, we use Theorem \ref{Bertini} and Theorem \ref{properties-of-ch_2}, (4)  to reduce (3) and (4) to the surface case where these properties follow from Proposition \ref{global-basic-properties-c_2}, (3) and (4). For example in (4), Theorem \ref{Bertini} says that for  general $H\in |L_1|$  both $H$ and $\pi^{-1}(H)\in |\pi^*L_1|$ are normal and irreducible, so we can reduce the statement to lower dimension.
\end{proof}

\medskip

One of the most important results that follow from our theory is the following Riemann--Roch type
inequality:

\begin{Theorem}\label{main-consequence}
Let  $L_1, ..., L_{n-2}$ be very ample line bundles on $X$. Then there exists a constant $C$ such that for all
$\cE\in \Refl(\cO_X)$ we have
\begin{align*}
&\left\lvert \chi (X, c_1(L_1)... c_1(L_{n-2}) \cdot \cE) -r\chi (X, c_1(L_1)... c_1(L_{n-2}) \cdot \cO_X) -  \int _X \ch _2 (\cE)L_1...L_{n-2}  \right. \\
 & \left. +\frac{1}{2}   c_1(\cE).(K_X+L_1+...+L_{n-2}). L_1...L_{n-2} 
\right\rvert \le C\cdot r^2,
    \end{align*}
where $r$ is the rank of $\cE$.
\end{Theorem}

\begin{proof}
Let $Y$ be the product of $(n-2)$ projective spaces $|L_i|$ and let us consider the incidence scheme $\cS\subset X\times _k Y$, whose points are  of the form $(x;H_1,..., H_{n-2})\in X\times _kY$ with $x\in \bigcap _{j\le n-2} H_j $.
Let $\bar\eta: \Spec K\to Y$ be a geometric generic point of $Y$. Then by Theorem \ref{Bertini}
the fiber $\cS_K$ of the projection $\cS\to Y$ over $\bar \eta$ is a normal surface contained in  $X_K$.
Note that for any $\cE \in \Refl(\cO_X)$  we have by Theorem \ref{RR-theorem-easy}
\begin{align*}
 \chi (X, c_1(L_1)... c_1(L_{n-2}) \cdot \cE) =&\chi (\cS _K, \cE _K)= \int_{\cS _K}\ch_2 (\cE _K)
 - \frac {1}{2}c_1 (\cE _K). K_{\cS _K}\\ 
&+r \chi (\cS _K, \cO_{\cS _K})+\sum _{x\in \Sing \cS _K} a(x, \cE _K).
\end{align*}
By adjunction we have $K_{\cS _K} =K_{X_K}+(L_1)_K+...+(L_{n-2})_K$.
Therefore $$c_1 (\cE _K). K_{\cS _K}= c_1(\cE).(K_X+L_1+...+L_{n-2}). L_1...L_{n-2}.$$
We have $\chi (\cS _K, \cO_{\cS _K})=\chi (X, c_1(L_1)... c_1(L_{n-2}) \cdot \cO_X).$
By Theorem \ref{properties-of-ch_2}, (4)
we also have $$\int_{\cS _K}\ch_2 (\cE _K)= \int _X \ch_2 (\cE) L_1...L_{n-2}.$$
Now the required inequality follows from Theorem \ref{bound-on-a(x,E)} applied to singularities of $\cS _K$.
\end{proof}

\medskip
 
\begin{Remark}\label{determinacy-c_2-higher-dim}
By definition of $\int_Xc_2$ we immediately have
$$\int_Yc_2 ( F_X ^{[*]}\cE) L_1...L_{n-2}= p^2 \cdot \int_Xc_2(\cE)L_1...L_{n-2}$$
 for all $\cE\in \Refl(\cO_X)$.
As in Theorem \ref{RR-theorem} $\int_Xc_2 : \Refl(\cO_X)\times N^1( X) ^{\times (n-2)}\to \RR$ is uniquely determined by 
this property and  inequality from Theorem \ref{main-consequence}.
\end{Remark}

\subsection{The second Chern class in characteristic zero}

Here we assume that $k$ is an algebraically closed field of characteristic $0$. In this subsection we assume that
$X$ is a normal projective variety  with at most quotient singularities in codimension $2$. This means that any general complete intersection surface in $X$ has  at most quotient singularities.

\begin{Proposition}\label{ch_2-char0}
Let $\cE$ be a reflexive coherent $\cO_X$-module of rank $r$ and let $L_1,..., L_{n-2}$ be line bundles on $X$.
Then the sequence
 $$\left(\frac {  \chi (X, c_1(L_1)...c_1(L_{n-2}) \cdot {\Sym ^{[m]} \cE} ) }{m^{r+1}} \right)_{m\in \NN}$$
 is convergent to some rational number. 
\end{Proposition}

\begin{proof}
Proof is similar to the proof of Proposition \ref{existence-of-limit}.
Namely, for any line bundle $M_1$ on $X$, Lemma \ref{Kollar}, (1) and Lemma \ref{reflexive-in-K(X)} imply that
$$c_1(M_1)c_1(L_1)...c_1(L_{n-2}) \cdot {\Sym^{[m]}\cE} =c_1(M_1)c_1(L_1)...c_1(L_{n-2}) \cdot 
(\cO_X^{\oplus \binom{m+r-1}{m} }-[\cO_{ -c_1(\Sym^{[m]}\cE )}] ) .$$
Since $c_1(\Sym^{[m]}\cE ) =\binom{m+r-1}{r}c_1 (\cE)$, Lemma \ref{Kollar} (1) and Lemma \ref{difference} give 
$$c_1(M_1)c_1(L_1)...c_1(L_{n-2}) \cdot  [\cO_{ -c_1(\Sym^{[m]}\cE )}] =-\binom{m+r-1}{r}c_1(M_1)c_1(L_1)...c_1(L_{n-2}) \cdot [\cO_{c_1 (\cE)}].$$
Therefore
$$ \lim_{m\to \infty} \frac {  \chi (X, c_1(M_1) c_1(L_1)...c_1(L_{n-2}) \cdot {\Sym ^{[m]} \cE} ) }{m^{r+1}} =0.$$
Now the same arguments at that in  the proof of Proposition \ref{existence-of-limit} reduce the assertion to the case when all $L_i$ are very ample. Similarly as before we reduce to the case when the base field $k$ is uncountable and
then restrict to a very general complete intersection surface $S\in |L_1|\cap ...\cap |L_{n-2}|$. Then we get
$$\lim _{m\to \infty }\frac {  \chi (X, c_1(L_1)...c_1(L_{n-2}) \cdot {\Sym^{[m]}\cE}  ) }{m^{r+1}}=\lim _{m\to \infty}\frac{\chi(S, \Sym^{[m]} (\cE |_S )) }{m^{r+1}},$$ 
which by Theorem \ref{asymptotic-RR-quotient-surfaces} exists and is a rational number.
\end{proof}

\medskip

The method of proof of Proposition \ref{ch_2-char0}
works for any normal projective variety $X$ in characteristic $0$ for which we know Conjecture \ref{Wahl's-conjecture} for any general complete intersection surface in $X$.
As in the previous subsections, the above proposition allows us to define $\int _X c_2 (\cE)L_1...L_{n-2}$ by
$$\int _X c_2 (\cE)L_1...L_{n-2}:= c_1(\cE)^2.L_1...L_{n-2}-r!  \lim_{m\to \infty} \frac {  \chi (X, c_1(L_1)...c_1(L_{n-2}) \cdot {\Sym ^{[m]} \cE} ) }{m^{r+1}} . $$ 
We can use this to define  $\int _X \ch _2 (\cE)L_1...L_{n-2}$ and $\int _X \Delta (\cE)L_1...L_{n-2}$. 

\begin{Theorem}\label{properties-of-ch_2-char-0}
	Let us fix $\cE\in \Refl(\cO_X) $.
The map $\int_X\ch_2 (\cE): N^1( X) ^{\times (n-2)}\to \QQ$, sending $(L_1,..., L_{n-2})$ to $ \int _X \ch _2 (\cE)L_1...L_{n-2}$,
satisfies the following properties:
\begin{enumerate}
\item It is $\ZZ$-linear and symmetric.
\item If $\cE$ is a vector bundle on $X$ then $$\int _X \ch _2 (\cE)L_1...L_{n-2}=\int_X \ch _2 (\cE)\cap c_1(L_1)\cap ...\cap c_1(L_{n-2})[X].$$
\item If $n>2$ and  $ L_1$ is very ample then for a very general hypersurface $H\in |L_1|$ we have
$$\int _X \ch _2 (\cE)L_1...L_{n-2}=\int _{H}\ch_2 (\cE|_{H})L_2|_{H}...L_{n-2}|_{H}.$$
\item There exists some $N\in \NN$ such that the image of $\int_X\ch_2$ is contained in $\frac{1}{N}\ZZ\subset \QQ$.
\end{enumerate}
\end{Theorem}

\begin{proof}
Properties (1)--(3) can be proven in the same way as Theorem \ref{properties-of-ch_2}. To prove (4) it is sufficient to note that this holds on normal surfaces with only quotient singularities and use the proof of Theorem \ref{main-consequence} to reduce to this case.
\end{proof}

\medskip

\begin{Remark}
Let $(X,D)$ be a projective klt pair defined over an algebraically closed field of characteristic zero. In that case $X$ is well known to have quotient singularities in codimension $2$.  Then for fixed $\cE\in \Refl(\cO_X)$, the multilinear forms $\Pic X ^{\times(n-2)}\to \QQ$ given by sending $(L_1,...,L_{n-2})$ to
$c_1 (\cE)^2.L_1...L_{n-2}$, $\int _X c_2 (\cE)L_1...L_{n-2}$ or $\int _X \ch_2 (\cE)L_1...L_{n-2}$ coincide (modulo passing to numerical equivalence classes) with analogous forms considered in \cite[Theorem 3.13]{GKPT}.
This follows, e.g., from  \cite[Theorem 5.1]{La0}  and \cite[(3.13.2)]{GKPT}. The construction of these forms  in 
\cite{GKPT} uses Mumford's Chern classes for $\QQ$-bundles on $\QQ$-varieties and it does not generalize to varieties that do not have quotient singularities in codimension $2$. Moreover, this construction works well only in the characteristic zero case.
\end{Remark}

\subsection{Properties of  Chern classes in characteristic zero}

In this subsection we keep the notation form previous subsection. Then we have the following proposition analogous to Proposition \ref{global-basic-properties-c_2}.

\begin{Proposition}\label{further-properties}
For any  line bundles $L_1,..., L_{n-2}$ on $X$ the following conditions are satisfied:
\begin{enumerate}
\item For any rank $1$ reflexive coherent $\cO_X$-module $\cL$ on $X$ we have $$\int _X c_2 (\cL)L_1...L_{n-2}=0.$$
\item For any $\cE\in \Refl(\cO_X) $  and any rank $1$ reflexive coherent $\cO_X$-module $\cL$ on $X$ we have
$$ \int_X \Delta (\cE \hat \otimes \cL)L_1...L_{n-2}=\int_X \Delta (\cE) L_1...L_{n-2}.$$
\item If $f: \tilde X\to X$ is a resolution of singularities then for any vector bundle $\cF$ on $\tilde X$ we have
$$\int_{\tilde X} \Delta ( \cF)f^*L_1...f^*L_{n-2} \ge  \int_{X} \Delta ( f_{[*]}\cF) L_1...L_{n-2}. $$\item If  $ \pi: Y\to X$ is a finite morphism from a normal  projective variety $Y$ then for any $\cE\in \Refl(\cO_X) $ we have
$$\int_Yc_2 ( \pi ^{[*]}\cE) \pi^*L_1...\pi^*L_{n-2}= \deg \pi \cdot \int_Xc_2(\cE)L_1...L_{n-2}.$$
\end{enumerate}
\end{Proposition}

\begin{proof}
The proof is the same as that of Proposition \ref{further-properties-char-p} except that our assumption on 
morphisms to have separably generated residue field extensions is automatically satisfied in characteristic zero.
\end{proof}

\medskip

\begin{Remark}
In case $k$ has characteristic $0$ and $f: Y\to X$ is a quasi-\'etale morphism of klt pairs the property (4) was proven in 
\cite[Lemma 3.16]{GKPT}. Note that our assertion is much stronger as it does not require $f$ to be quasi-\'etale.
\end{Remark}

\medskip

In characteristic $0$ the Riemann--Roch type inequality analogous to that from Theorem \ref{main-consequence} is also satisfied but we will not use it in the following.

\subsection{K-theoretic formulation}\label{K-theory-reflexive}

Let $X$ be a normal projective  variety defined over an algebraically closed field $k$.
We define the \emph{Grothendieck group $K^{\refl }(X)$  of reflexive sheaves on $X$} as 
the free abelian group on the isomorphism classes $[\cE]$ of coherent reflexive $\cO_X$-modules modulo
the relations $[\cE_2]=[\cE_1]+[\cE_3]$ for each locally split short exact sequence
$$0\to \cE_1\to \cE_2\to \cE_3\to 0$$
of coherent reflexive $\cO_X$-modules. 
Clearly, there is a natural group homomorphism $K^{\refl }(X)\to K(X)$ sending a class of a reflexive sheaf to the class of the same sheaf. This homomorphism is surjective as any coherent sheaf on a normal projective variety has a finite resolution by reflexive sheaves. However, it is in general not injective as a short exact sequence of reflexive sheaves does not need to be locally split.

Note that $K^{\refl }(X)$ is a ring with unity  $[\cO_X]$ and
multiplication given by
$$[\cE_1]\cdot [\cE_2]=[\cE_1\hat{\otimes} \cE_2].$$
So we can think of $K^{\refl }(X)$ as an analogue of the Grothendieck ring of vector bundles on a smooth variety. 
We have well defined  $\ZZ$-linear maps $\ch_1: K^{\refl}(X)\to A^1(X)$, $[\cE]\to c_1(\cE)$ 
and $\chi: K^{\refl } (X)\to \ZZ$, $[\cE]\to \chi (X, \cE)$.

Let us assume that  $k$ has positive characteristic $p$. Since the functor $F_X^{[*]}$ is exact on locally split short exact sequences, we have a well defined homomorphism of rings
$$F_X^{[*]}:  K^{\refl }(X)\to  K^{\refl } (X)$$
sending $[\cE]$ to $[F_X^{[*]}\cE]$ (here we use that in the definition of $K^{\refl }(X)$ we consider only locally split short exact sequences). If $n\ge 2$ and we fix some line bundles $L_1,...,L_{n-2}$ then
$\int_X\ch_2 (\cdot)L_1...L_{n-2}$ defines a $\ZZ$-linear map $K^{\refl }(X)\to \RR$ (this can be proven 
as Corollary \ref{a-of-sum-char-p}; see \cite[Lemma 2.1]{La4}). 

If $X$ is a surface then we define the Chern character
$$\ch: K^{\refl }(X)\to A^*(X)\otimes \RR$$
by setting $\ch ([\cE]):=\rk \cE +c_1(\cE)+\ch_2(\cE)$.  Corollary \ref{a-of-sum-char-p}
and our definitions imply that this extends to a homomorphism of abelian groups. Note that $A^*(X)\otimes \RR$ is a ring and conjecturally, $\ch$ is also a homomorphism of rings (see Remark \ref{behaviour-on-tensors}).

\section{Applications}

Let $X$ be a normal projective variety of dimension $n$ defined over an algebraically closed field $k$.
Let $\cO_X (1)$ be an ample line bundle on $X$.  All the results below hold when we replace $\cO_X(1)$ by a collection of ample line bundles
but proofs become much more complicated and we deal with these results in \cite{La4}.

\medskip

\subsection{Boundedness on normal varieties}

We will often write $H$ for a Cartier divisor such that $\cO_X (1)=\cO_X(H)$. 
Let $\cE$ be a torsion free coherent  $\cO_X$-module of rank $r$. 
We will  write $H^{i}\cdot \cE $ for the class $c_1 (\cO_X(1) )^{i}\cdot \cE $ in $K(X)$.
By \cite[Chapter VI, Theorem 2.13]{Ko}
 we have
$$\chi (X, \cE (m))= \sum _{i=0}^n \chi (X, H^{i}\cdot \cE  )\binom{m+i-1}{i}.$$
There are also uniquely determined integers $a_0^H(\cE),..., a_n^H (\cE)$ such that
$$\chi (X, \cE (m))= \sum _{i=0}^n a_i^H(\cE)\binom{m+n-i}{n-i}.$$

\begin{Lemma}\label{formula-for-a_2}
We have
 $a_0^H(\cE)=rH^n$,
$$a_1^H(\cE)= c_1(\cE).H^{n-1}-\frac{r}{2}(K_X+(n-1)H).H^{n-1}-rH^n$$
and
$$a_2^H(\cE)=\chi (X, H^{n-2}\cdot  \cE )-c_1(\cE).H^{n-1}+\frac{r}{2}(K_X+(n-1)H).H^{n-1}.$$
\end{Lemma}

\begin{proof}
Using  Theorem \ref{Bertini} and  \cite[Lemma 1.1.12 and Corollary 1.1.14]{HL}) we can construct a sequence  of divisors $H_1,..., H_n\in |\cO_X (1)|$ such that for all $i=1,...,n$ the following conditions are satisfied:
\begin{enumerate}
\item the intersection  $\bigcap _{j\le i} H_j $ is normal,
\item the restriction $\cE |_{\bigcap _{j\le i} H_j }$ is torsion free of rank $r$. 
\end{enumerate}
From the definition we see that
$a_0 ^H(\cE)= \chi (\cE |_{\bigcap _{j\le n} H_j} )$
and
$$a_i ^H(\cE)= \chi (\cE |_{\bigcap _{j\le n-i} H_j} )- \chi (\cE |_{\bigcap _{j\le n-i+1} H_j} )$$
for $i>0$.  So the equality  $a_0^H(\cE)=rH^n$ is clear.  
By the above the curve $C:={\bigcap _{j\le n-1} H_j} $ is smooth and by the adjunction formula we have 
$$-2\chi (\cO_C)=\deg \omega _C=\deg \cO_C(K_X+H_1+...+H_{n-1})=(K_X+(n-1)H).H^{n-1}.$$
This, together with the Riemann--Roch theorem for $\cE |_C$, gives the formula for  $a_1^H(\cE)$.
 The last formula follows easily from the previous two formulas and equality 
 $a_2 ^H(\cE)= \chi (\cE |_{\bigcap _{j\le n-2} H_j} )- (a_0^H(\cE)+a_1^H(\cE))$.
\end{proof}

\medskip

We define a \emph{slope} of $\cE$ with respect to $H$ as
$$\mu _H(\cE) :=\frac{c_1(\cE).H^{n-1}}{r}.$$
This allows us to define slope $H$-semistability and the maximal $H$-destabilizing slope $\mu _{\max, H} (\cE)$ 
for any coherent torsion free  $\cO_X$-module $\cE$.

\medskip

Let us recall the following  special case of \cite[Theorem 4.4]{La1}.

\begin{Theorem} \label{boundedness-1}
Let $(X, H)$ be as above and let us fix some integers  $a_0$, $a_1$, $a_2$ and a rational number $\mu _{\max}$.
Then the set of coherent reflexive  $\cO_X$-modules $\cE$ with $a_0 ^H(\cE)=a_0$, $a_1^H (\cE)=a_1$,
$a_2 ^H(\cE)\ge a_2$ and $\mu _{\max ,H } (\cE)\le \mu _{\max}$ is bounded.
\end{Theorem}

Using Lemma \ref{formula-for-a_2} we can rewrite the above theorem in the following way:

\begin{Corollary}\label{boundedness-2}
Let  us fix some positive integer $r$, integers $c_1$ and $\chi$ and some rational 
number $\mu _{\max}$. Then the set of coherent reflexive $\cO_X$-modules $\cE$ of rank $r$ with
$c_1 (\cE). H^{n-1}=c_1$,  $\chi (X, H^{n-2}\cdot \cE )\ge \chi$ 
and $\mu _{\max, H} (\cE)\le \mu _{\max}$ is bounded. In particular, the set of slope $H$-semistable coherent reflexive  $\cO_X$-modules $\cE$ of rank $r$ with $c_1 (\cE). H^{n-1}=c_1$
and  $\chi (X, H^{n-2}\cdot \cE )\ge \chi$ 
is bounded.
\end{Corollary}

\medskip

\subsection{Boundedness and Bogomolov's inequality on normal varieties in positive characteristic}

In this subsection we assume that the base field $k$ has positive characteristic. 
In the following, we write $N_H(X)$ for the group
$N_L(X)$ introduced in Subsection \ref{intersecting-2-Weil} in the case of $(L_1,...,L_{n-2})=(\cO_X (1),...,\cO_X (1))$.

\begin{Theorem} \label{boundedness-3}
Let us fix some positive integer $r$  and some real
numbers $c_2$ and $\mu _{\max}$. Let us also fix some class $c_1\in N_H(X)$.
Then the set $\cA$ of coherent reflexive $\cO_X$-modules $\cE$ of rank $r$ with $[c_1(\cE)]=c_1\in N_H(X)$,
$\int _X c_2 (\cE) H^{n-2}\le c_2$  and $\mu _{\max ,H } (\cE)\le \mu _{\max}$
is bounded.
\end{Theorem}

\begin{proof}
By Theorem \ref{main-consequence} we have
\begin{align*}
&\chi (X, H^{n-2} \cdot \cE) \le \frac{1}{2}   c_1.(c_1-(K_X+(n-2)H)). H^{n-2} 
 -\int _X c _2 (\cE)H^{n-2}   \\
 &  +r\chi (X, H^{n-2} \cdot \cO_X) + C\cdot r.
    \end{align*}
So our assertion follows from Corollary \ref{boundedness-2}.
\end{proof}

\medskip

The proof of the following theorem was motivated by a similar proof by Maruyama in the smooth case (see the proof of 
\cite[Corollary 2.10]{Ma}).

\begin{Theorem}\label{Bogomolov's-inequality}
Let us fix some positive integer $r$ and some non-negative rational number $\alpha$.
There exists some constant $\tilde C=\tilde C(X,H,r,\alpha)$ depending only on $X$, $H$, $r$ and $\alpha$ such that for every coherent
reflexive    $\cO_X$-module $\cE$ of rank $r$ with
$\mu _{\max ,H } (\cE)-\mu _H (\cE)\le \alpha$ we have
$$\int _X \Delta (\cE) H^{n-2}\ge \tilde C.$$
\end{Theorem}

\begin{proof}
Let us choose a basis $L_1,....,L_s$ of $N_H (X)$ as a $\ZZ$-module (see Lemma \ref{finite-generation})
and let us write $[c_1 (\cE)]= \sum a_i L_i$ for some $a_i\in \ZZ$. There exist  uniquely determined integers $q_i$
and $r_i$ such that $a_i=q_i r+r_i$ and $0\le r_i <r$.
Since 
$$\int _X \Delta (\cE) H^{n-2}= \int _X \Delta (\cE (-\sum q_i L_i)) H^{n-2}$$
and there are only finitely many possibilities for  $[c_1 (\cE (-\sum q_i L_i))]= \sum r_i L_i$, it is sufficient to prove existence of the above constant assuming that $c_1=[c_1(\cE)]\in N_H (X)$ is fixed.

If $\int _X c_2 (\cE) H^{n-2}\ge 0$ then  $\int _X \Delta (\cE) H^{n-2}\ge - (r-1)c_1^2.H^{n-2}$.

On the other hand, by Theorem \ref{boundedness-3}  the set $\cA$ of reflexive coherent $\cO_X$-modules $\cE$ of rank $r$ with $[c_1(\cE)]=c_1\in N_H(X)$, $\int _X c_2 (\cE) H^{n-2}\le 0$  and $$\mu _{\max ,H } (\cE)\le \mu _{\max}= \alpha+\frac{1}{r}c_1.H^{n-1}$$ is bounded. So there exists some constant $D$ such that for every $\cE\in \cA$ we have
and any $\cE$-regular sequence $H_1,..., H_{n-2}\in |\cO_X (1)|$ we have
$$\chi (X, H^{n-2}\cdot \cE)=\chi (X, \cE |_{\bigcap _{j\le n-2} H_j })\le D.$$
But Theorem \ref{main-consequence} gives
\begin{align*}
\chi (X, H^{n-2} \cdot \cE) \ge &\frac{1}{2r}   c_1.(c_1-r(K_X+(n-2)H)). H^{n-2} 
 -\frac{1}{2r}\int _X \Delta (\cE)H^{n-2}   \\
 &  +r\chi (X, H^{n-2} \cdot \cO_X) - C\cdot r^2.
    \end{align*}
Therefore for any $\cE\in \cA$ we have
\begin{align*}
\int _X \Delta (\cE) H^{n-2}\ge \max &(- (r-1)c_1^2.H^{n-2} ,c_1 . (c_1 -r(K_X+(n-2)H).H^{n-2}\\
&+ 2r^2\chi (X, H^{n-2} \cdot \cO_X )-2C\cdot r^3 +2D\cdot r ).
\end{align*}
\end{proof}

The above theorem easily implies Bogomolov's inequality for strongly semistable reflexive sheaves:

\begin{Corollary}\label{Bogomolov-strongly-ss}
Let $\cE$ be a coherent reflexive $\cO_X$-module. If $\cE$ is strongly slope $H$-semistable
then 
$$\int _X \Delta (\cE)H^{n-2}\ge 0.$$
\end{Corollary}

\begin{proof}
Let us consider the set $\{ F_X^{[m]}\cE \} _{m\ge 0}$. Note that by assumption  each sheaf in this set is
reflexive of rank $r$ and slope $H$-semistable, i.e., $\mu _{\max ,H } (F_X^{[m]}\cE)-\mu _H (F_X^{[m]}\cE)=0$.
So by Theorem \ref{Bogomolov's-inequality} there exists a constant $\tilde C$ such that
$$\int _X \Delta (F_X^{[m]}\cE) H^{n-2}  =p^{2m} \int _X \Delta (\cE) H^{n-2}\ge \tilde C$$
for all $m\ge 0$. Dividing by $p^{2m}$ and passing with $m$ to infinity, we get the required inequality.
\end{proof}

\begin{Corollary} \label{boundedness-4}
Let us fix some positive integer $r$, integer $\ch_1$ and some real
numbers $\ch _2$ and $\mu _{\max}$. 
Then the set $\cB$ of coherent reflexive $\cO_X$-modules $\cE$ of rank $r$ with $\int _X \ch_1 (\cE). H^{n-1}=\ch_1$, $\int _X \ch_2 (\cE). H^{n-2}\ge \ch_2$  and $\mu _{\max ,H } (\cE)\le \mu _{\max}$ is bounded.
\end{Corollary}

\begin{proof} 
By Theorem \ref{Bogomolov's-inequality} there exists a constant $\tilde C$ 
such that for all  $\cE \in \cB$ we have
$$\tilde C\le \int _X \Delta (\cE) H^{n-2}=  c_1 (\cE)^2.H^{n-2} - 2r \int _X \ch_2 (\cE). H^{n-2} .$$
Therefore $c_1 (\cE)^2.H^{n-2}\ge \tilde C+2r \,\ch _2$. Let us write $[c_1(\cE)]=\alpha [H]+D\in N_H (X)$, where $\alpha = \frac{\ch _1}{H^n}$.
Then $D.H^{n-1}=0$ and $c_1 (\cE)^2.H^{n-2}=\alpha ^2 H^n+D^2.H^{n-2}$ and we have
$$D^2.H^{n-2}\ge \tilde C+2r \,\ch _2 - \alpha ^2 H^n.$$
But by the Hodge index theorem (see Lemma \ref{finite-generation}) the intersection form is negative definite on $H^{\perp}\subset N_H(X)$, so there are only finitely many possibilities for $D$ and hence there are also finitely many possibilities for the classes $[c_1(\cE)]\in N_H (X)$. Now the assertion follows from Theorem \ref{boundedness-3}.
\end{proof}

\subsection{Application to F-divided sheaves on normal varieties}

In this section we use the above developed theory to reprove \cite[Theorem 2.1]{ES} in the style well-known from the smooth varieties.

\medskip

Let $X$ be a normal projective variety of dimension $\ge 1$ defined over an algebraically closed field $k$.
Let $j: U\hookrightarrow X$ be an open embedding of a big open subset contained in the regular locus of $X$.
Let $\EE= (\cE_m , \sigma _m)_{m\ge 0}$ be an F-divided sheaf on $U$, i.e., $\cE_m$ are coherent $\cO_X$-modules and $\sigma_m: F_X^* \cE_{m+1} \to \cE_m$ are isomorphisms of $\cO_X$-modules. Let us set $\tilde \cE _m:= j_*\cE_m$.

\begin{Lemma}
For every ample divisor $H$ we have  $ [c_1 (\tilde \cE _{m})]=0\in N_H(X)$ for all $m$ and $$ \lim _{m\to \infty }\int _{X}c_2 (\tilde \cE _{m})H^{n-2} =0.$$
\end{Lemma}

\begin{proof}
The isomorphisms $\sigma _m: \cE _m\to F_X^{*}\cE _{m+1}$ extend uniquely to
 isomorphisms $\tilde \cE _m\to F_X^{[*]}\tilde \cE _{m+1}$.
In particular, we see that $c_1 (\tilde \cE _{0})= p^{m}  c_1(\tilde \cE _{m})$. By Lemma \ref{finite-generation}
there exists some $N$ such that for every Weil divisor $D$ on $X$ we have
 $$N\cdot c_1 (\tilde \cE _{0}).D.H^{n-2}= p^{m}  N \cdot c_1(\tilde \cE _{m}).D.H^{n-2}\in p^{m}\ZZ .$$
Since $N_H(X)$ is a free $\ZZ$-module, the class of $ c_1 (\tilde \cE _{0})$ in $N_H(X)$  must vanish. This also implies vanishing of all classes $ [c_1 (\tilde \cE _{m})]\in N_H(X)$.
The second assertion follows from equalities
 $\int _Xc_2 (\tilde \cE _{0}) H^{n-2}= p^{2m} \int _Xc_2 (\tilde \cE _{m})H^{n-2}$.
\end{proof}

\medskip

The following corollary is due to Esnault and Srinivas (see \cite[Theorem 2.1]{ES} for a slightly weaker statement).

\begin{Corollary}
There exists a bounded set $\cS$ of slope $H$-semistable sheaves such that for every $r>0$ and every rank $r$ F-divided sheaf 
$\EE= (\cE_m , \sigma _m)_{m\ge 0}$ on $U$ there exists  an integer $m_0(\EE){\ge 0}$ such that for all $m\ge m_0(\EE)$ 
 the sheaves $\cE_m$ lie in the set $\cS$.
\end{Corollary}

\begin{proof}
Let us fix  a positive real number $c_2$ and let  $\cS$ denote the set of slope $H$-semistable sheaves $\cF$ on $X$
such that  $ [c_1 ( \cF )]=0\in N_H(X)$ and $\int _{X}c_2 ( \cF )H^{n-2}\le c_2$. By Theorem  \ref{boundedness-3} this set is bounded. 

Now fix $\EE= (\cE_m , \sigma _m)_{m\ge 0}$ on $U$ and as before set $\tilde \cE _m:=j_*\cE_m$.
Since $$\mu _{\max, H} (\tilde \cE_m)\ge p \mu _{\max, H} ( \tilde \cE_{m+1}),$$ we have
$\mu _{\max, H} ( \tilde \cE_0)\ge p ^m\mu _{\max, H} ( \tilde \cE_{m})$.  Since $r!\cdot \mu _{\max, H} ( \tilde \cE_{m})\in \ZZ$, there exists $m_1(\EE)\in \ZZ _{\ge 0}$ such that for all $m\ge m_1(\EE)$ we have $\mu _{\max, H} ( \tilde \cE_{m})\le 0$. Since  $\mu _{H} ( \tilde \cE_{m})= 0$ we see that
the sheaves $\tilde \cE _m=j_*\cE_m$ are slope $H$-semistable  for all $m\ge m_1(\EE)$.
Now the above lemma implies that for any $\EE$ there exists  $m_0(\EE)\in \ZZ _{\ge 0}$  such that $\cE_m\in \cS$ for all $m\ge m_0(\EE)$.
\end{proof}

\begin{Remark}
The above proof shows that one can choose $m_0$ that does not depend on $\EE$ but only on the rank $r$ if and only if  $\int _{X}c_2 (\tilde \cE _{m})H^{n-2} =0$ for all $m\ge 0$. In general, this would follow if we knew that the second relative Chern classes of vector bundles defined in Subsection \ref{local-relative-RR} are rational with bounded denominators but this seems very unlikely.  However, this happens, e.g., if $X$ has only quotient singularities in codimension $2$ (cf. Theorem \ref{properties-of-ch_2-char-0}, (5)).
\end{Remark}

\section*{Acknowledgements}

The author would like to thank M. Enokizono for pointing out \cite{En}.

A large part of the paper was written while the author was an External Senior Fellow at Freiburg Institute for Advanced Studies (FRIAS), University of Freiburg, Germany. 
The author would like to thank Stefan Kebekus for his hospitality during the author's stay in FRIAS.

The  author was partially supported by Polish National Centre (NCN) contract numbers 2018/29/B/ST1/01232 and 2021/41/B/ST1/03741. 
The research leading to these results has received funding from the European Union's Horizon 2020 research and innovation programme under the Maria Sk{\l}o\-do\-wska-Curie grant agreement No 754340.

\end{document}